\documentclass[a4paper,11pt,reqno]{amsart}
\usepackage{amsmath,amssymb,amsthm}
\usepackage[expansion=false]{microtype}
\usepackage{listings}
\usepackage[margin=2.5cm]{geometry}
\usepackage[colorlinks=true,linkcolor=black,citecolor=black,urlcolor=black]{hyperref}
\usepackage[nameinlink,capitalize]{cleveref}
\usepackage{xurl}
\hypersetup{pdftitle={Maximum 2-scattered subspaces of V(r,q\textasciicircum 6) for every r >= 3},
  pdfauthor={Daniele Bartoli, Francesco Ghiandoni, Alessandro Giannoni, Giuseppe Marino, Alessandro Neri},
  pdfkeywords={finite fields, scattered subspaces, linear sets, linearised polynomials, MRD codes}}
\allowdisplaybreaks
\newcommand{\F}{\mathbb F}
\newcommand{\Tr}{\operatorname{Tr}}
\newcommand{\N}{\operatorname{N}}
\newcommand{\GL}{\operatorname{GL}}
\newcommand{\Fix}{\operatorname{Fix}}
\newtheorem{theorem}{Theorem}[section]
\newtheorem{lemma}[theorem]{Lemma}
\newtheorem{proposition}[theorem]{Proposition}
\newtheorem{corollary}[theorem]{Corollary}
\theoremstyle{remark}
\newtheorem{remark}[theorem]{Remark}
\numberwithin{equation}{section}
\title[Maximum $2$-scattered subspaces of $V(r,q^6)$]{Maximum $2$-scattered subspaces of $V(r,q^6)$ for every $r\ge3$}
\author[D. Bartoli]{Daniele Bartoli}
\address{Daniele Bartoli: Dipartimento di Matematica e Informatica, Universit\`a degli Studi di Perugia, Perugia, Italy}
\email{daniele.bartoli@unipg.it}
\author[F. Ghiandoni]{Francesco Ghiandoni}
\address{Francesco Ghiandoni: Department of Mathematics, FAMNIT, University of Primorska, Koper, Slovenia}
\email{francesco.ghiandoni@famnit.upr.si}
\author[A. Giannoni]{Alessandro Giannoni}
\address{Alessandro Giannoni: Dipartimento di Matematica e Applicazioni ``Renato Caccioppoli'', Universit\`a degli Studi di Napoli Federico II, Napoli, Italy}
\email{alessandro.giannoni@unina.it}
\author[G. Marino]{Giuseppe Marino}
\address{Giuseppe Marino: Dipartimento di Matematica e Applicazioni ``Renato Caccioppoli'', Universit\`a degli Studi di Napoli Federico II, Napoli, Italy}
\email{giuseppe.marino@unina.it}
\author[A. Neri]{Alessandro Neri}
\address{Alessandro Neri: Dipartimento di Matematica e Applicazioni ``Renato Caccioppoli'', Universit\`a degli Studi di Napoli Federico II, Napoli, Italy}
\email{alessandro.neri@unina.it}
\date{}
\begin{document}
\begin{abstract}
We prove that for every prime power $q$ and every integer $r\ge3$
there exists a maximum $2$-scattered $\F_q$-subspace of $V(r,q^6)$,
that is, a $2r$-dimensional $\F_q$-subspace meeting every
$2$-dimensional $\F_{q^6}$-subspace in dimension at most $2$.
Since an example in dimension three is known and direct sums preserve
the property, it suffices to treat $r=4$ and $r=5$. For $r=4$ we study
the subspaces
$U_b=\{(x,x^q+bx^{q^5}):\Tr_{q^{12}/q^4}(x)=0\}$ of
$\F_{q^{12}}^2$, $b\in\F_{q^{12}}$. We describe the parameters $b$ for
which $U_b$ is not maximum $2$-scattered in terms of trace conditions
and Dickson determinants, and deduce that at least
$\frac12q^{12}-q^{11}+O(q^{10})$ parameters work; this number is
positive for $q\ge3$, and a computer search settles $q=2$.
For every $r\ge5$ coprime to $6$ we show that
$U_s=\{x+\varepsilon s\,x^{q^r}:x\in\F_{q^{2r}}\}\subseteq\F_{q^{6r}}$
is maximum $2$-scattered for all
$\varepsilon\in\F_{q^6}\setminus\F_{q^2}$ and
$s\in\F_{q^r}\setminus\F_q$. Consequently, $\F_{q^6}$-linear
$[2r,r,4]_{q^6/q}$ MRD codes exist for every $r\ge3$; for $r$ coprime
to $6$, the codes obtained from $U_s$ are equivalent to their duals.
\end{abstract}
\maketitle

\section{Introduction}\label{sec:introduction}

Let $V(r,q^n)$ denote an $r$-dimensional vector space over $\F_{q^n}$.
For $1\le h\le r-1$, an $\F_q$-subspace $U$ of $V(r,q^n)$ is called $h$-scattered if it
spans $V(r,q^n)$ over $\F_{q^n}$ and meets every $h$-dimensional
$\F_{q^n}$-subspace in an $\F_q$-subspace of dimension at most $h$.
For $h=1$, this is the scattered property with respect to the
Desarguesian spread obtained by field reduction. Blokhuis and Lavrauw
\cite{BL2000} proved that a scattered subspace of $V(r,q^n)$ has
dimension at most $rn/2$. Ball, Blokhuis and Lavrauw
\cite{BBL2000} constructed maximum scattered subspaces for
$(r,n)=(3,4)$. Constructions of maximum scattered linear sets and their applications
to complete caps were developed by Bartoli, Giulietti, Marino and
Polverino \cite{BGMP2018}. Csajb\'ok, Marino, Polverino and Zullo
\cite{CMPZ2017} proved existence in the remaining cases when $rn$ is
even and extended the connection with maximum rank distance (MRD)
codes established by Sheekey \cite{Sheekey2016}.

The condition for $h>1$ was introduced by Csajb\'ok, Marino, Polverino
and Zullo in \cite{CMPZ2021}. They proved that, for $n\ge h+1$, an
$h$-scattered subspace has dimension at most $rn/(h+1)$
\cite[Theorem~2.3]{CMPZ2021}. We call a subspace attaining this bound
maximum. Their constructions attain the bound whenever $h+1$ divides
$r$ and $n\ge h+1$, by taking direct sums of subspaces scattered with respect to
hyperplanes. They also introduced Delsarte duality for these subspaces
and used it to obtain existence results in some cases where $h+1$ does
not divide $r$. A related geometric construction was studied by Napolitano,
Polverino, Zini and Zullo \cite{NPZZ2021}, who introduced linear sets
of $h$-pseudoregulus type through projections of Desarguesian spreads.
They characterised the maximum $h$-scattered members of this class
using Moore exponent sets and determined their equivalence.

Many scattered subspaces are described as graphs of linearised
polynomials. For a graph
$U_f=\{(x,f(x)):x\in\F_{q^n}\}$, where $f$ is linearised,
the scattered condition says that $f(x)y=f(y)x$ forces
$x,y$ to be dependent over $\F_q$. Bartoli and Zhou \cite{BZ2018}
studied exceptional scattered polynomials, for which scatteredness
persists over infinitely many extensions of the field of coefficients.
Bartoli, Marino, Neri and Vicino \cite{BMNV2024} introduced
$h$-scattered sequences of multivariate linearised polynomials,
together with exceptionality and indecomposability, and constructed
an infinite family of exceptional scattered sequences of order two.
Bartoli, Giannoni and Marino \cite{BGM2025} subsequently constructed
infinite families of exceptional scattered sequences of every order
$m\ge3$, which under additional arithmetic conditions give scattered
subspaces that are not direct sums of smaller ones. In \cite{BGM2026}, the same authors constructed a further
infinite family of maximum $h$-scattered subspaces of
$V(m(h+1),q^n)$ and studied the associated MRD codes, including
generalised rank weights which distinguish them from earlier families.
These constructions concern ambient dimensions divisible by $h+1$;
the cases $(r,n,h)=(4,6,2)$ and $(5,6,2)$ needed here fall
outside that range.

The $h$-scattered property is the case $k=h$, together with the
spanning condition, of the notion of $(h,k)_q$-evasive subspace, that
is, a subspace meeting every $h$-dimensional $\F_{q^n}$-subspace in
dimension at most $k$ over $\F_q$; bounds and constructions for
evasive subspaces are given in \cite{BCMT2021}. Sheekey and Van de
Voorde \cite{SV2020} related linear sets, including those scattered
with respect to hyperplanes, to rank-metric codes and their duality.
Zini and Zullo \cite{ZZ2021} showed that maximum $h$-scattered
subspaces correspond to MRD codes and characterised them through
their intersections with hyperplanes, and Marino, Neri and Trombetti
\cite{MNT2023} interpreted this correspondence in terms of
generalised rank weights and used it to study near MRD and
quasi-MRD codes. Maximum $h$-scattered subspaces have been used to
construct sets with few intersection numbers with hyperplanes
\cite{AMSZ2023}, and they are closely related to subspace designs
\cite{SZ2024} and, through $h$-designs, to maximum sum-rank distance
codes \cite{SS2025}.

We write $\Tr_{q^a/q^c}$ and $\N_{q^a/q^c}$ for the trace and
norm from $\F_{q^a}$ to $\F_{q^c}$ whenever $c$ divides $a$.
All fields are regarded as subfields of a fixed algebraic closure
of $\F_q$.

The case $(r,n,h)=(4,6,2)$ asks for an eight-dimensional
$\F_q$-subspace of $V(4,q^6)$ whose intersection with every
two-dimensional $\F_{q^6}$-subspace has dimension at most two over
$\F_q$. Bartoli, Ghiandoni, Giannoni and Marino constructed such
subspaces when $q$ is an odd power of $2$ in \cite{BGGM2025}.
One description of their construction is
\[
 \{(x,y,x^{q^2}+y^q,x^q+y^{q^3}):x,y\in\F_{q^6},
       \ \Tr_{q^6/q^2}(x)=\Tr_{q^6/q^2}(y)=0\}.
\]
Together with the dimension-three construction and direct sums,
their result gives maximum $2$-scattered subspaces of $V(r,q^6)$
for $r\ge3$, $r\ne5$, under the same restriction on $q$.

In this paper we construct maximum $2$-scattered subspaces of
$V(4,q^6)$ and $V(5,q^6)$ for every prime power $q$. In dimension
four we use graphs of the binomials $x^q+bx^{q^5}$ restricted to the
kernel of $\Tr_{q^{12}/q^4}$; in dimension five, and more generally in
every dimension $r\ge5$ coprime to $6$, we use a two-dimensional
$\F_{q^r}$-subspace of $\F_{q^{6r}}$. In dimension three one can take
$\{(x,x^q,x^{q^2}):x\in\F_{q^6}\}$ \cite[Example~2.5]{CMPZ2021}.
Together with direct sums, this gives the following result.

\begin{theorem}\label{thm:main}
For every prime power $q$ and every integer $r\ge3$, there exists a
maximum $2$-scattered $\F_q$-subspace of $V(r,q^6)$, of dimension
$2r$ over $\F_q$.
\end{theorem}
\begin{proof}
Write $r=3k+r_0$ with $k\ge0$ and $r_0\in\{3,4,5\}$. For $r_0=3,4,5$
take the subspace in \cite[Example~2.5]{CMPZ2021}, \cref{thm:four}
and \cref{thm:coprime}, respectively. Its direct sum with $k$ copies
of the three-dimensional example is maximum $2$-scattered in
$V(r,q^6)$ by \cite[Theorem~2.4]{CMPZ2021}.
\end{proof}

The cases $r=4$ and $r=5$ are covered by the following two results.

For the dimension-four construction, put
\begin{equation}\label{eq:family}
 W_0=\ker\Tr_{q^{12}/q^4},\qquad
 g_b(x)=x^q+bx^{q^5},\qquad
 U_b=\{(x,g_b(x)):x\in W_0\},
 \quad b\in\F_{q^{12}}.
\end{equation}
The ambient space $\F_{q^{12}}^2$ has dimension four over
$\F_{q^6}$. Since $W_0$ has dimension two over $\F_{q^4}$ and
the first coordinate of the graph map is $x$, every $U_b$ has
dimension eight over $\F_q$.

\begin{theorem}\label{thm:four}
For every prime power $q$ there is $b\in\F_{q^{12}}$ such that $U_b$
is a maximum $2$-scattered subspace of $V(4,q^6)$.
For every $q\ge3$, the number of such parameters is at least
\begin{equation}\label{eq:bound}
q^{12}
-\frac{q^4+2q^3+q+2}{2}\,(q^8+q^4+1)
-(q^2+1)(q^4-1)q^4.
\end{equation}
This bound is positive for every $q\ge3$. For $q=2$, the exhaustive
computation in \cref{app:q2} gives exactly $1056$ such parameters.
\end{theorem}

The lower bound is $\frac12q^{12}-q^{11}+O(q^{10})$ as
$q\to\infty$, so the proportion of parameters $b$ for which
$U_b$ is maximum $2$-scattered has $\liminf$ at least $1/2$.
Note that \eqref{eq:bound} counts parameters $b$, not equivalence
classes of the subspaces $U_b$.
% q=3 data below come from an independent computation; re-check before submission.
For $q=3$ the bound \eqref{eq:bound} equals $1631$, whereas an
exhaustive computation shows that $201936$ of the $3^{12}=531441$
parameters are good, including $4200$ of the $6643$ parameters with
$\N_{3^{12}/3^4}(b)=-1$. For instance, if
$\F_{3^{12}}=\F_3[\omega]$ with
$\omega^{12}+2\omega^{11}+\omega^{10}+\omega^9+\omega^8+2=0$,
then $b=\omega^2+1$ is good.

For dimensions coprime to six, consider instead
\begin{equation}\label{eq:coprime-family}
 U_s=\{x+\varepsilon s x^{q^r}:x\in\F_{q^{2r}}\}
       \subseteq\F_{q^{6r}},
\end{equation}
where $\varepsilon\in\F_{q^6}\setminus\F_{q^2}$ and
$s\in\F_{q^r}\setminus\F_q$. The ambient field has dimension
$r$ over $\F_{q^6}$.

\begin{theorem}\label{thm:coprime}
Let $q$ be a prime power and let $r\ge5$ satisfy $\gcd(r,6)=1$.
For every $\varepsilon\in\F_{q^6}\setminus\F_{q^2}$ and every
$s\in\F_{q^r}\setminus\F_q$, the subspace $U_s$ in
\eqref{eq:coprime-family} is a maximum $2$-scattered
$\F_q$-subspace of $V(r,q^6)$, of dimension $2r$ over $\F_q$.
\end{theorem}

In particular, \cref{thm:coprime} covers $r=5$, the only case of
\cref{thm:main} not obtained from dimensions three and four.
The corresponding $[2r,r,4]_{q^6/q}$ MRD codes are equivalent to
their duals; see \cref{prop:isodual}. More generally,
\cref{thm:main} gives $\F_{q^6}$-linear $[2r,r,4]_{q^6/q}$ MRD
codes for every $r\ge3$. For $r>3$ their length $2r$ exceeds the
extension degree $6$. In this range the Gabidulin construction is not
available, and $\F_{q^m}$-linear MRD codes correspond to maximum
$h$-scattered subspaces \cite{ZZ2021,MNT2023}; see also
\cite{BMN2023}.

Both proofs rest on the hyperplane characterisation of maximum
$h$-scattered subspaces \cite[Corollary~5.2]{ZZ2021} and on the
correspondence between $\F_q$-subspaces and rank-metric codes
\cite{Randrianarisoa2020,MNT2023}. For a $2r$-dimensional subspace
$U$ of $V(r,q^6)$ they reduce maximum $2$-scatteredness to the
requirement that every nonzero $\F_{q^6}$-linear functional has rank
at least four on $U$ (\cref{prop:rank-code}).

For $U_b$ we bound the number of parameters $b$ for which some
functional has rank at most three, distinguishing according to the
$\F_{q^2}$-span of its image. If this span is contained in a plane,
then $b$ is determined by three auxiliary parameters, which we count.
If the span is $\F_{q^6}$, the image is the fixed space of a
$q$-semilinear involution of $\F_{q^6}$, and the condition on $b$ is
the vanishing of a Dickson determinant. Up to scalar multiplication
there are $q^4+q$ such fixed spaces; a fixed space and its orthogonal
complement with respect to $(x,y)\mapsto\Tr_{q^6/q}(xy)$ exclude the
same parameters, and at most $2q^3$ classes are mapped to themselves.
Pairing the remaining classes gives the bound in \cref{thm:four}.

For $U_s$, a trace-adjoint argument turns the rank condition into
the requirement that two explicit three-dimensional
$\F_{q^r}$-subspaces of $\F_{q^{6r}}$ be complementary, one pair for
each three-dimensional $\F_q$-subspace $C$ of $\F_{q^6}$. If
$\F_{q^2}C=\F_{q^6}$, a failure would make $s^2$ a root of a nonzero
polynomial over $\F_q$ of degree at most three; otherwise, after
rescaling $C$, it would force $s^2\in\F_q$. Both are impossible,
because every element of $\F_{q^r}\setminus\F_q$ has degree at least
five over $\F_q$ when $\gcd(r,6)=1$.

\Cref{sec:preliminaries} contains the common rank criterion
and trace facts. \Cref{sec:rank,sec:bad,sec:count,sec:orthogonality}
develop the dimension-four construction and prove
\cref{thm:four}. \Cref{sec:coprime,sec:complementarity}
give the construction in dimensions coprime to six and prove
\cref{thm:coprime}. \Cref{sec:codes} treats the associated
codes and their duality. \Cref{sec:conclusions} lists open
problems, and \cref{app:q2} contains the computation for $q=2$.

\section{Hyperplane intersections and trace maps}\label{sec:preliminaries}

This section contains the rank criterion used for both constructions
and two facts on subspaces of $\F_{q^6}$.

An $\F_q$-subspace $U$ of $V(r,q^n)$ defines the linear set
\[
 L_U=\{\langle u\rangle_{\F_{q^n}}:u\in U\setminus\{0\}\}
 \subseteq\operatorname{PG}(r-1,q^n).
\]
Its rank is $\dim_{\F_q}U$. The weight of a projective subspace
$\operatorname{PG}(H)$ with respect to $U$ is
$\dim_{\F_q}(U\cap H)$. If $\dim_{\F_q}U=t$, then
\[
 |L_U|\le\frac{q^t-1}{q-1},
\]
with equality if and only if every point has weight one. In this
case $L_U$ and $U$ are called scattered. See \cite{Polverino2010}
for the general theory of linear sets. Recall that, following
\cite{CMPZ2021}, our definition of an $h$-scattered subspace includes
the requirement that it spans $V(r,q^n)$.

In the coding-theoretic part of this section, $n$ denotes the length
of a code and $m$ the degree of the field extension; in our
applications $m=6$ and $n=2r$.
For a vector $v=(v_1,\ldots,v_n)\in\F_{q^m}^n$, put
\[
 \operatorname{wt}_q(v)
 =\dim_{\F_q}\langle v_1,\ldots,v_n\rangle_{\F_q}.
\]
An $[n,k,d]_{q^m/q}$ code is a $k$-dimensional
$\F_{q^m}$-linear subspace of $\F_{q^m}^n$ whose minimum nonzero
rank weight is $d$. The rank-metric Singleton bound is
\begin{equation}\label{eq:singleton}
 mk\le\min\{m(n-d+1),\ n(m-d+1)\}.
\end{equation}
A code attaining equality is called maximum rank distance, or MRD.
The dual of an MRD code with respect to the standard
$\F_{q^m}$-bilinear form is again MRD. For $n\le m$ this is in
\cite{Gabidulin1985}; in general it follows from Delsarte's theorem
\cite{Delsarte1978}, since the standard dual of an $\F_{q^m}$-linear
code corresponds to the Delsarte dual of the associated matrix code
\cite{Ravagnani2016}.

Let $U\le\F_{q^m}^k$ have dimension $n$ over $\F_q$ and span
$\F_{q^m}^k$ over $\F_{q^m}$. Choose an $\F_q$-basis
$u_1,\ldots,u_n$ of $U$ and put
\[
 \mathcal C_U=\{(\ell(u_1),\ldots,\ell(u_n)):
       \ell\in\operatorname{Hom}_{\F_{q^m}}
                    (\F_{q^m}^k,\F_{q^m})\}.
\]
The spanning condition makes the evaluation map injective, so
$\mathcal C_U$ has dimension $k$ over $\F_{q^m}$. We regard the
$u_i$ as column vectors and the codewords as row vectors, so that
the matrix $G$ with columns $u_1,\ldots,u_n$ generates
$\mathcal C_U$. Replacing the basis of $U$ multiplies $G$ on the
right by a matrix in $\GL_n(q)$, which preserves rank weights.
This is the code-system correspondence of
\cite{Randrianarisoa2020}; see also \cite[Section~2]{MNT2023}.

The following is the case needed here of the hyperplane
characterisation in \cite[Corollary~5.2]{ZZ2021} and the
correspondence in \cite[Theorem~4.9]{MNT2023}; the hypothesis
$n\ge h+3$ of \cite[Corollary~5.2]{ZZ2021} holds, since $n=6$ and
$h=2$. For completeness, we include a proof.

\begin{proposition}\label{prop:rank-code}
Let $r\ge3$, and let $U$ be a $2r$-dimensional $\F_q$-subspace
spanning $V(r,q^6)$ over $\F_{q^6}$. The following are equivalent:
\begin{enumerate}
\item $U$ is maximum $2$-scattered;
\item $\dim_{\F_q}\ell(U)\ge4$ for every nonzero
      $\F_{q^6}$-linear functional $\ell$;
\item $\mathcal C_U$ is a $[2r,r,4]_{q^6/q}$ MRD code.
\end{enumerate}
Equivalently, every $\F_{q^6}$-hyperplane meets $U$ in dimension
at most $2r-4$ over $\F_q$. Moreover, condition~(2) itself implies
the spanning condition.
\end{proposition}
\begin{proof}
Let $G$ have as columns an $\F_q$-basis of $U$.
A nonzero word $v$ of $\mathcal C_U^\perp$ satisfies
$Gv^{\mathsf T}=0$. If its rank weight is $t\le3$, a change of
coordinates over $\F_q$, together with the corresponding change
of basis of $U$, puts this relation in the form
\[
 \alpha_1u_1+\cdots+\alpha_tu_t=0,
\]
where $u_1,\ldots,u_t$ are $\F_q$-independent and the
$\alpha_i\in\F_{q^6}$ are $\F_q$-independent. The case $t=1$ cannot
occur, since $u_1\ne0$.
Conversely, any nontrivial $\F_{q^6}$-linear relation among at
most three $\F_q$-independent elements of $U$ gives a nonzero
word of rank at most three in the dual, after extending those
elements to a basis of $U$.

The $2$-scattered condition is equivalent to every three
$\F_q$-independent elements of $U$ being $\F_{q^6}$-independent.
It also excludes a dependent pair: such a pair lies on an
$\F_{q^6}$-line, and the spanning condition supplies an element
of $U$ outside that line, giving three $\F_q$-independent
elements in an $\F_{q^6}$-plane. Thus condition~(1) is equivalent
to $d(\mathcal C_U^\perp)\ge4$.

Both codes have length $2r$ and dimension $r$ over $\F_{q^6}$.
By \eqref{eq:singleton},
\[
 6r\le2r(7-d),
\]
so $d\le4$. For $d=4$ and $r\ge3$, the right-hand side of
\eqref{eq:singleton} is $\min\{6(2r-3),6r\}=6r$.
Hence a code with these parameters is MRD if and only if its minimum
distance is $4$. Since duality preserves the MRD property, (1) and (3)
are equivalent.

The rank weight of the word defined by $\ell$ is
$\dim_{\F_q}\ell(U)$. This proves the equivalence of (2) and
(3). If $H=\ker\ell$, then
\[
 \dim_{\F_q}(U\cap H)=2r-\dim_{\F_q}\ell(U),
\]
which gives the hyperplane formulation. Finally, if $U$ does
not span the ambient space, some nonzero functional vanishes
on $U$, contradicting (2).
\end{proof}

For a finite extension $\F_{q^a}/\F_{q^c}$, we use the
nondegenerate trace pairing
\[
 (x,y)\longmapsto\Tr_{q^a/q^c}(xy).
\]
Orthogonal complements are taken with respect to the pairing
specified in each application. The adjoint of a linear map
between spaces with nondegenerate pairings has the same rank
as the map, and the orthogonal complement of its kernel is the
image of its adjoint. These facts follow by writing the
pairings in bases: the adjoint matrix is obtained from the
transpose by multiplying by invertible Gram matrices.

\begin{lemma}\label{lem:trace-planes}
Every two-dimensional $\F_{q^m}$-subspace of $\F_{q^{3m}}$
is a scalar multiple of $\ker\Tr_{q^{3m}/q^m}$. In particular,
any two such subspaces are scalar multiples of one another.
\end{lemma}
\begin{proof}
Every $\F_{q^m}$-linear functional on $\F_{q^{3m}}$ is
$x\mapsto\Tr_{q^{3m}/q^m}(ax)$ for a unique
$a\in\F_{q^{3m}}$, by nondegeneracy of the trace pairing.
For $a\ne0$, its kernel is
$a^{-1}\ker\Tr_{q^{3m}/q^m}$, and every plane is the
kernel of such a functional.
\end{proof}

\begin{lemma}\label{lem:fixed-spaces}
A three-dimensional $\F_q$-subspace $V\le\F_{q^6}$ spans
$\F_{q^6}$ over $\F_{q^2}$ if and only if it is the fixed
space of a $q$-semilinear involution over $\F_{q^2}$.
This involution is uniquely determined by $V$.
\end{lemma}
\begin{proof}
Suppose first that $\F_{q^2}V=\F_{q^6}$.
An $\F_q$-basis $v_1,v_2,v_3$ of $V$ is then an
$\F_{q^2}$-basis of $\F_{q^6}$. Put
\[
 \rho(\lambda_1v_1+\lambda_2v_2+\lambda_3v_3)
 =\lambda_1^qv_1+\lambda_2^qv_2+\lambda_3^qv_3
 \qquad(\lambda_i\in\F_{q^2}).
\]
This is the unique $q$-semilinear map fixing each $v_i$; in
particular, every $q$-semilinear involution with fixed space $V$
equals $\rho$. It satisfies $\rho^2=1$ and $\Fix\rho=V$.

Conversely, let $\rho$ be a $q$-semilinear involution and put
$V=\Fix\rho$. Choose
$\kappa\in\F_{q^2}\setminus\F_q$. If $v=\kappa w$ with
$v,w\in V$, then $\kappa w=\rho(v)=\kappa^qw$, so $w=0$.
Thus $V\cap\kappa V=\{0\}$ and $\dim_{\F_q}V\le3$.
For every $z\in\F_{q^6}$, both $z+\rho(z)$ and
$\kappa^qz+\kappa\rho(z)$ belong to $V$, and
\[
 z=\frac{(\kappa^qz+\kappa\rho(z))-\kappa(z+\rho(z))}
          {\kappa^q-\kappa}.
\]
Hence $\F_{q^2}V=\F_{q^6}$, giving
$\dim_{\F_q}V\ge3$ and proving the claim.
\end{proof}

In both constructions, we distinguish whether the
$\F_{q^2}$-span of a three-dimensional $\F_q$-subspace of
$\F_{q^6}$ has dimension two or three. In the first case we use
\cref{lem:trace-planes}; for the trace-kernel family,
\cref{lem:fixed-spaces} describes the subspaces in the second case.

\section{The trace-kernel family}\label{sec:rank}

We use the notation of \eqref{eq:family}.
The trace $\Tr_{q^{12}/q^4}$ is a nonzero $\F_{q^4}$-linear map
onto $\F_{q^4}$, so $W_0$ has dimension two over $\F_{q^4}$.

For $\alpha,\beta\in \F_{q^{12}}$ write
\[
 c_{\alpha,\beta}(x)=\Tr_{q^{12}/q^6}
       (\alpha x+\beta g_b(x)).
\]
We call $b$ \emph{bad} if $U_b$ is not maximum $2$-scattered, and
\emph{good} otherwise. \Cref{prop:rank-code}, with $r=4$, gives
\begin{equation}\label{eq:rankcriterion}
 b\text{ is bad}\quad\Longleftrightarrow\quad
 \dim_{\F_q}c_{\alpha,\beta}(W_0)\le3
 \text{ for some }(\alpha,\beta)\ne(0,0).
\end{equation}
In fact, every $\F_{q^6}$-hyperplane of $\F_{q^{12}}^2$ is the kernel of
$(x,y)\mapsto\Tr_{q^{12}/q^6}(\alpha x+\beta y)$ for a nonzero pair,
by nondegeneracy of the trace pairing, and its intersection with $U_b$
has dimension $8-\dim_{\F_q}c_{\alpha,\beta}(W_0)$.
The required upper bound on hyperplane intersections is $4$.
It also implies that $U_b$ spans $\F_{q^{12}}^2$ over $\F_{q^6}$.

We need two facts on linearised polynomials. Every $\F_q$-linear map $\F_{q^{12}}\to \F_{q^{12}}$ has a unique
expression $\sum_{i=0}^{11}a_i x^{q^i}$. A polynomial of this form
vanishes on $W_0$ if and only if
\begin{equation}\label{eq:restriction}
 a_j=a_{j+4}=a_{j+8}\qquad(0\le j\le3).
\end{equation}
Indeed, a linear map vanishing on $W_0$ factors through
$\Tr_{q^{12}/q^4}$. Every $\F_q$-linear map from $\F_{q^4}$ to
$\F_{q^{12}}$ has a unique expression
$z\mapsto\sum_{j=0}^3a_jz^{q^j}$. Composing with the trace gives
the coefficient relations in \eqref{eq:restriction}.

Here exponents $q^i$ are read with $i$ modulo $12$ (or modulo $6$ for
maps of $\F_{q^6}$); since $6$ and $12$ are even, the parity of $i$ is
well defined. In particular, if a sum of an even-exponent polynomial
and an odd-exponent polynomial vanishes on $W_0$, then both vanish
on $W_0$.

An $\F_{q^m}$-linear map is singular if it has a nonzero kernel.
We use the following determinant criterion.
\begin{lemma}[Dickson's criterion; see e.g.\ {\cite{WL2013}}]\label{lem:dickson}
An $\F_{q^m}$-linear map
$\psi(z)=a_0z+a_1z^{q^m}+a_2z^{q^{2m}}$ on $\F_{q^{3m}}$
is singular if and only if
\begin{equation}\label{eq:dickson}
 \N_{q^{3m}/q^m}(a_0)+\N_{q^{3m}/q^m}(a_1)
 +\N_{q^{3m}/q^m}(a_2)
 -\Tr_{q^{3m}/q^m}(a_0a_1^{q^m}a_2^{q^{2m}})=0.
\end{equation}
\end{lemma}
\begin{proof}
Let $\mathcal{D}=(a_{j-i}^{q^{mi}})_{0\le i,j\le2}$, with coefficient
indices read modulo $3$. For an $\F_{q^m}$-basis $y_1,y_2,y_3$,
\[
 \mathcal{D}\,(y_j^{q^{mi}})_{i,j}=(\psi(y_j)^{q^{mi}})_{i,j}.
\]
The matrix on a basis is invertible: otherwise a nonzero linearised
polynomial of degree at most $q^{2m}$ would vanish on the basis and
therefore on all $q^{3m}$ field elements. Hence the matrix on the
right is singular if and only if the images $\psi(y_j)$ are
dependent, that is, if and only if $\psi$ is singular.
If $A\in\F_{q^m}^{3\times3}$ is the matrix of $\psi$ in the basis
$y_1,y_2,y_3$, then $(\psi(y_j)^{q^{mi}})_{i,j}=(y_j^{q^{mi}})_{i,j}A$,
so $\det\mathcal{D}=\det A$ is the determinant of $\psi$ as an
$\F_{q^m}$-linear map. Expanding $\det\mathcal{D}$ gives
\eqref{eq:dickson}.
\end{proof}

For an $\F_{q^2}$-linear map $u(z)=u_0z+u_1z^{q^2}+u_2z^{q^4}$ of
$\F_{q^6}$ we write $\det_{\F_{q^2}}(u)$ for its determinant as an
$\F_{q^2}$-linear map. By the proof of \cref{lem:dickson} with $m=2$,
\[
 \det_{\F_{q^2}}(u)=\N_{q^6/q^2}(u_0)+\N_{q^6/q^2}(u_1)
   +\N_{q^6/q^2}(u_2)-\Tr_{q^6/q^2}(u_0u_1^{q^2}u_2^{q^4}).
\]

\section{Characterisation of the bad parameters}\label{sec:bad}

Write an $\F_{q^2}$-linear map on $\F_{q^6}$ as
\[
 u(z)=u_0z+u_1z^{q^2}+u_2z^{q^4},\qquad u_0,u_1,u_2\in \F_{q^6},
\]
and put $\rho_u(z)=u(z^{q^3})=u_2z^q+u_0z^{q^3}+u_1z^{q^5}$.
Expanding the composition shows that
$\rho_u^2=1$ is equivalent to the three equations
\begin{equation}\label{eq:involution}
\begin{aligned}
 u_0u_0^{q^3}+u_1u_2^{q^5}+u_2u_1^q&=1,\\
 u_0u_1^{q^3}+u_1u_0^{q^5}+u_2u_2^q&=0,\\
 u_0u_2^{q^3}+u_1u_1^{q^5}+u_2u_0^q&=0.
\end{aligned}
\end{equation}

Define
\begin{equation}\label{eq:Lambda}
 \Lambda_{u,b}(y)=u_0(1-b^{q^3})y-u_2y^{q^4}
                         +u_1b^{q^{11}}y^{q^8},
\end{equation}
and
\begin{equation}\label{eq:Delta}
\begin{aligned}
 \Delta_b(u)={}&\N_{q^{12}/q^4}(1-b^{q^3})\N_{q^6/q^2}(u_0)
       +\N_{q^{12}/q^4}(b^{q^3})\N_{q^6/q^2}(u_1)\\
 &-\N_{q^6/q^2}(u_2)
 +\Tr_{q^{12}/q^4}
   \bigl((1-b^{q^3})b^{q^7}u_0u_1^{q^2}u_2^{q^4}\bigr).
\end{aligned}
\end{equation}

\begin{theorem}\label{thm:bad}
Let $b\in\F_{q^{12}}$ satisfy $\N_{q^{12}/q^4}(b)\ne-1$.
The $\F_{q^4}$-linear map $z\mapsto z+(zb)^{q^8}$ of $\F_{q^{12}}$ is
bijective by \cref{lem:bad-alphabeta}; let $\beta_0$ be the preimage
of $1$, so that
\[
 \beta_0+(\beta_0b)^{q^8}=1,
\]
and put $\eta=\beta_0^{q^{11}}$.
Then $b$ is bad if and only if at least one of the following holds:
\begin{enumerate}
\item There are $\varepsilon\in \F_{q^6}\setminus\F_{q^2}$,
      $\nu\in\F_{q^4}^*$ and $\tau\in\F_{q^4}$ such that
\begin{equation}\label{eq:planecondition}
 \varepsilon+\nu\bigl(\varepsilon^q+
       \eta(\varepsilon^{q^5}-\varepsilon^q)\bigr)=\tau.
\end{equation}
\item There are $u_0,u_1,u_2\in \F_{q^6}$ satisfying
      \eqref{eq:involution} and $\Delta_b(u)=0$.
\end{enumerate}
In the first condition, $\varepsilon$ can be restricted to any set of
$q^2+1$ representatives of the sets
$\{\kappa\varepsilon+\mu:\kappa\in\F_{q^2}^*,\ \mu\in\F_{q^2}\}$,
$\varepsilon\in\F_{q^6}\setminus\F_{q^2}$, that is, of the points of
the projective line $\mathrm{PG}(\F_{q^6}/\F_{q^2})$, where
$\F_{q^6}/\F_{q^2}$ denotes the quotient $\F_{q^2}$-vector space.
\end{theorem}

We prove \cref{thm:bad} at the end of this section, after three
lemmas describing the possible images $c_{\alpha,\beta}(W_0)$ of
dimension at most $3$. Note that $c_{\alpha,\beta}(W_0)$ has the same
dimension as $c_{\alpha/\gamma,\beta/\gamma}(W_0)$ for every
$\gamma\in\F_{q^6}^*$.

\begin{lemma}\label{lem:bad-alphabeta}
Let $b\in\F_{q^{12}}$ satisfy $\N_{q^{12}/q^4}(b)\ne-1$. Then the maps
$z\mapsto z+(zb)^{q^8}$ and $g_b$ are bijective on $\F_{q^{12}}$.
Moreover, $\dim_{\F_q}c_{\alpha,\beta}(W_0)\ge4$ for every pair
$(\alpha,\beta)\ne(0,0)$ with $\alpha\beta=0$.
\end{lemma}
\begin{proof}
The $\F_{q^4}$-linear map
$z\mapsto z+(zb)^{q^8}$ is bijective: a nonzero kernel
element would satisfy $z=-(zb)^{q^8}$, and taking norms to $\F_{q^4}$
and cancelling $\N_{q^{12}/q^4}(z)$ would give
$1=-\N_{q^{12}/q^4}(b)$, contrary to the hypothesis.
In particular, $\beta_0$ in \cref{thm:bad} is well defined and nonzero.

The map $g_b$ is bijective on $\F_{q^{12}}$. For $b\ne0$, a nonzero root would
give $(x^{q^4-1})^q=-b^{-1}$, and taking the norm to $\F_{q^4}$
would give $\N_{q^{12}/q^4}(b)=-1$. The case $b=0$ is immediate.
Since $g_b$ is $q$-semilinear over $\F_{q^4}$, its image of $W_0$
is an $\F_{q^4}$-plane. By \cref{lem:trace-planes}, every such plane is a scalar multiple
of $W_0$.

The kernel of $\Tr_{q^{12}/q^6}$ is an $\F_{q^6}$-line. Its intersection
with any scalar multiple of $W_0$ has dimension at most $2$ over
$\F_{q^2}$. Otherwise the entire $\F_{q^6}$-line would be contained in
that multiple of $W_0$. After rescaling, this would say that
$a\F_{q^6}\subseteq W_0$ for some $a\ne0$, whereas the nonzero polynomial
\[
 e\longmapsto ae+a^{q^4}e^{q^4}+a^{q^8}e^{q^2}
\]
cannot vanish at all $q^6$ elements of $\F_{q^6}$, since its degree is $q^4<q^6$.
Thus the trace has rank at least $8-4=4$ on every scalar multiple
of $W_0$. Apply this to $\alpha W_0$ when $\beta=0$, and to
$\beta g_b(W_0)$ when $\alpha=0$.
\end{proof}

\begin{lemma}\label{lem:bad-plane}
Let $b$, $\beta_0$ and $\eta$ be as in \cref{thm:bad}. There is a
pair $(\alpha,\beta)$ with $\alpha\beta\ne0$ such that
$c_{\alpha,\beta}(W_0)$ has dimension at most $3$ over $\F_q$ and is
contained in an $\F_{q^2}$-plane if and only if condition~(1) of
\cref{thm:bad} holds. In condition~(1), $\varepsilon$ can be
restricted to any set of representatives as in \cref{thm:bad}.
\end{lemma}
\begin{proof}
Suppose that $\alpha\beta\ne0$, that
$\dim_{\F_q}c_{\alpha,\beta}(W_0)\le3$, and that
$c_{\alpha,\beta}(W_0)$ is contained in an
$\F_{q^2}$-plane. By \cref{lem:trace-planes}, every such plane is a scalar multiple
of $\ker\Tr_{q^6/q^2}$. Thus there is $\gamma\in\F_{q^6}^*$ such that,
setting $\alpha'=\alpha/\gamma$ and $\beta'=\beta/\gamma$, we have
\[
 c_{\alpha',\beta'}(W_0)
 =\gamma^{-1}c_{\alpha,\beta}(W_0)
 \subseteq\ker\Tr_{q^6/q^2}.
\]
Equivalently, for every $x\in W_0$,
\[
\begin{aligned}
 0&=\Tr_{q^6/q^2}\bigl(c_{\alpha',\beta'}(x)\bigr)\\
  &=\Tr_{q^{12}/q^2}(\alpha'x)
    +\Tr_{q^{12}/q^2}(\beta'x^q+\beta'bx^{q^5}).
\end{aligned}
\]
By \eqref{eq:restriction}, the even and odd parts must vanish
separately on $W_0$. For $a\in\F_{q^{12}}$, the polynomial
$\Tr_{q^{12}/q^2}(ax)=\sum_{j=0}^5a^{q^{2j}}x^{q^{2j}}$
vanishes on $W_0$ if and only if $a=a^{q^4}$, that is,
$a\in\F_{q^4}$. Also, invariance of the trace under the $q^8$-th
power gives
\[
 \Tr_{q^{12}/q^2}(\beta'x^q+\beta'bx^{q^5})
 =\Tr_{q^{12}/q^2}\bigl((\beta'+(\beta'b)^{q^8})x^q\bigr).
\]
Since $x\mapsto x^q$ permutes $W_0$, the two vanishing conditions
are therefore equivalent to
\begin{equation}\label{eq:plane-alpha-beta}
 \alpha'\in\F_{q^4},\qquad
 \beta'+(\beta'b)^{q^8}\in\F_{q^4}.
\end{equation}

The bijectivity and $\F_{q^4}$-linearity in \cref{lem:bad-alphabeta}
show that the second condition holds if and only if
$\beta'=f\beta_0$ for some $f\in\F_{q^4}$.

The space $\ker\Tr_{q^6/q^2}$ has dimension $4$ over $\F_q$, whereas $c_{\alpha',\beta'}(W_0)$ has
dimension at most $3$. We show that there exists
$\varepsilon\in\F_{q^6}\setminus\F_{q^2}$ such that
\[
 \Tr_{q^6/q}\bigl(\varepsilon c_{\alpha',\beta'}(x)\bigr)=0
 \qquad\text{for every }x\in W_0.
\]

Let
\[
 \Theta=\bigl\{\varepsilon\in\F_{q^6}:
   \Tr_{q^6/q}(\varepsilon c_{\alpha',\beta'}(x))=0
   \text{ for every }x\in W_0\bigr\}.
\]
By nondegeneracy of the trace pairing, its dimension over $\F_q$ is
$6-\dim_{\F_q}c_{\alpha',\beta'}(W_0)$.
Moreover, $\F_{q^2}\subseteq\Theta$, since for
$\varepsilon\in\F_{q^2}$ and $x\in W_0$,
\[
 \Tr_{q^6/q}(\varepsilon c_{\alpha',\beta'}(x))
 =\Tr_{q^2/q}\bigl(\varepsilon
       \Tr_{q^6/q^2}(c_{\alpha',\beta'}(x))\bigr)=0.
\]
Consequently, $\dim_{\F_q}c_{\alpha',\beta'}(W_0)\le3$ if and only if
$\Theta$ contains an element outside $\F_{q^2}$.
For such an $\varepsilon$,
\[
\begin{aligned}
 \Tr_{q^6/q}\bigl(\varepsilon c_{\alpha',\beta'}(x)\bigr)
 &=\Tr_{q^{12}/q}\bigl(
      \alpha'\varepsilon x+\beta'\varepsilon x^q
        +\beta'b\varepsilon x^{q^5}\bigr)\\
 &=\Tr_{q^{12}/q}\Bigl(
     \bigl(\alpha'\varepsilon
       +(\beta')^{q^{11}}\varepsilon^{q^5}
       +(\beta'b)^{q^7}\varepsilon^q\bigr)x\Bigr).
\end{aligned}
\]

The annihilator of $W_0$ under the pairing
$(a,x)\mapsto\Tr_{q^{12}/q}(ax)$ is $\F_{q^4}$.
Indeed, $\F_{q^4}$ is contained in this annihilator by transitivity
of the trace, and both have dimension four over $\F_q$.
Consequently,
\[
 \alpha'\varepsilon+(\beta')^{q^{11}}\varepsilon^{q^5}
       +(\beta'b)^{q^7}\varepsilon^q\in\F_{q^4}.
\]
Substitute $\beta'=f\beta_0$. Since $f\in\F_{q^4}$,
we have $f^{q^{11}}=f^{q^7}=f^{q^3}$. Moreover, raising
$\beta_0+(\beta_0b)^{q^8}=1$ to the $q^{11}$-th power gives
$(\beta_0b)^{q^7}=1-\eta$. The preceding condition reads
\[
 \alpha'\varepsilon+f^{q^3}
   \bigl(\varepsilon^q+\eta(\varepsilon^{q^5}-\varepsilon^q)\bigr)
   \in\F_{q^4}.
\]
Here $\alpha'f\ne0$ because $\alpha\beta\ne0$.
Dividing by $\alpha'\in\F_{q^4}^*$ gives
\eqref{eq:planecondition}, with $\nu=f^{q^3}/\alpha'$.
Conversely, given \eqref{eq:planecondition}, choose
$f\in\F_{q^4}^*$ with $f^{q^3}=\nu$ and take
$\alpha'=1$, $\beta'=f\beta_0$. Then
$c_{1,f\beta_0}(W_0)\subseteq\ker\Tr_{q^6/q^2}$, and $\Tr_{q^6/q}(\varepsilon c_{\alpha',\beta'}(x))$ vanishes for that
$\varepsilon\notin\F_{q^2}$. Hence $\Theta$ has dimension at least
$3$, and by the dimension argument above,
$\dim_{\F_q}c_{1,f\beta_0}(W_0)\le3$. Since $f\beta_0\ne0$, the pair
$(1,f\beta_0)$ has the required properties.

For the assertion about representatives, put
$Y(\varepsilon)=\varepsilon^q+\eta(\varepsilon^{q^5}-\varepsilon^q)$.
If $\kappa\in\F_{q^2}^*$ and $\mu\in\F_{q^2}$, then
\[
 Y(\kappa\varepsilon+\mu)=\kappa^qY(\varepsilon)+\mu^q.
\]
If \eqref{eq:planecondition} holds, then
\[
 (\kappa\varepsilon+\mu)+\kappa^{1-q}\nu Y(\kappa\varepsilon+\mu)
 =\kappa\tau+\mu+\kappa^{1-q}\nu\mu^q\in\F_{q^4}.
\]
Thus replacing $\varepsilon$ by $\kappa\varepsilon+\mu$ and $\nu$ by
$\kappa^{1-q}\nu$ preserves the condition, and the change is
invertible. These sets are the points of the projective line
$\mathrm{PG}(\F_{q^6}/\F_{q^2})$, so there are $q^2+1$ of them.
\end{proof}

Now let $V\le\F_{q^6}$ be a three-dimensional $\F_q$-subspace with
$\F_{q^2}V=\F_{q^6}$. By \cref{lem:fixed-spaces}, $V$ is the fixed
space of a unique $q$-semilinear involution $\rho$ over $\F_{q^2}$.
The map $u(z)=\rho(z^{q^3})$ is $\F_{q^2}$-linear, since
$u(\lambda z)=\rho(\lambda^qz^{q^3})=\lambda u(z)$ for
$\lambda\in\F_{q^2}$. Write
$u(z)=u_0z+u_1z^{q^2}+u_2z^{q^4}$. Then $\rho=\rho_u$,
and $\rho^2=1$ gives \eqref{eq:involution}.
Conversely, coefficients satisfying \eqref{eq:involution}
define an involution whose fixed space has dimension three
and spans $\F_{q^6}$ over $\F_{q^2}$, by the same lemma.

\begin{lemma}\label{lem:bad-fullspan}
Let $u_0,u_1,u_2\in\F_{q^6}$ satisfy \eqref{eq:involution}, and let
$b\in\F_{q^{12}}$. The map $(\alpha,\beta)\mapsto\beta^{q^3}$ is a
bijection from the set of pairs $(\alpha,\beta)$ with
$c_{\alpha,\beta}(W_0)\subseteq\Fix\rho_u$ onto $\ker\Lambda_{u,b}$.
In particular, a nonzero such pair exists if and only if
$\Delta_b(u)=0$.
\end{lemma}
\begin{proof}
Put $\rho=\rho_u$. The condition
$c_{\alpha,\beta}(W_0)\subseteq\Fix\rho$ means that
\[
 \rho\bigl(c_{\alpha,\beta}(x)\bigr)=c_{\alpha,\beta}(x)
 \qquad\text{for every }x\in W_0.
\]
Recall that
\[
 c_{\alpha,\beta}(x)
 =\Tr_{q^{12}/q^6}(\alpha x)
  +\Tr_{q^{12}/q^6}(\beta g_b(x)).
\]
The first summand involves $x$ and $x^{q^6}$; the second involves
$x^q$, $x^{q^5}$, $x^{q^7}$ and $x^{q^{11}}$.
Composition with
$\rho_u(z)=u_2z^q+u_0z^{q^3}+u_1z^{q^5}$ interchanges even and odd
Frobenius indices. By \eqref{eq:restriction}, the even and odd parts
of an equality on $W_0$ must hold separately, so the condition is
equivalent to the two identities
\[
\begin{aligned}
 \rho\bigl(\Tr_{q^{12}/q^6}(\beta g_b(x))\bigr)
     &=\Tr_{q^{12}/q^6}(\alpha x),\\
 \rho\bigl(\Tr_{q^{12}/q^6}(\alpha x)\bigr)
     &=\Tr_{q^{12}/q^6}(\beta g_b(x))
 \qquad(x\in W_0).
\end{aligned}
\]
The second follows from the first by applying $\rho$ and using
$\rho^2=1$, so it suffices to impose the first identity.
Expanding it and applying \eqref{eq:restriction} gives
\[
\begin{aligned}
 u_2(\beta b)^{q^7}+u_1\beta^{q^{11}}-\alpha
 &=u_0\beta^{q^3}+u_1(\beta b)^{q^{11}}\\
 &=u_2\beta^{q^7}+u_0(\beta b)^{q^3}.
\end{aligned}
\]
The other coefficient triple is the $q^6$-th power of these
equations, so the condition is
\begin{equation}\label{eq:alpha-Lambda}
\begin{gathered}
 \Lambda_{u,b}(\beta^{q^3})=0,\\
 \alpha=u_2(\beta b)^{q^7}+u_1\beta^{q^{11}}
       -u_0\beta^{q^3}-u_1(\beta b)^{q^{11}}.
\end{gathered}
\end{equation}

For fixed $u$, these equations give a bijection between
$\ker\Lambda_{u,b}$ and the pairs $(\alpha,\beta)$ with
$c_{\alpha,\beta}(W_0)\subseteq\Fix\rho_u$: recover $\beta$ uniquely
from $\beta^{q^3}$ and then use the displayed formula for $\alpha$.
The zero kernel vector gives only $(\alpha,\beta)=(0,0)$.
Hence a nonzero pair exists if and only if $\Lambda_{u,b}$ is
singular. Applying \cref{lem:dickson} to its three coefficients
\[
 u_0(1-b^{q^3}),\qquad -u_2,\qquad u_1b^{q^{11}}
\]
gives $\Delta_b(u)$ in \eqref{eq:Delta}; here
$\N_{q^{12}/q^4}(u_i)=\N_{q^6/q^2}(u_i)$ and
$u_i^{q^8}=u_i^{q^2}$.
\end{proof}

\begin{proof}[Proof of \cref{thm:bad}]
By \cref{lem:bad-alphabeta}, $\beta_0$ is well defined, and every
pair $(\alpha,\beta)\ne(0,0)$ with
$\dim_{\F_q}c_{\alpha,\beta}(W_0)\le3$ satisfies $\alpha\beta\ne0$.
If the $\F_{q^2}$-span of $c_{\alpha,\beta}(W_0)$ has dimension at
most $2$, then $c_{\alpha,\beta}(W_0)$ is contained in an
$\F_{q^2}$-plane, and \cref{lem:bad-plane} gives condition~(1).
Otherwise $c_{\alpha,\beta}(W_0)$ spans $\F_{q^6}$ over $\F_{q^2}$;
since its dimension over $\F_q$ is at most $3$, it is a
three-dimensional subspace $V$ with $\F_{q^2}V=\F_{q^6}$. Writing
$V=\Fix\rho_u$ as above, \cref{lem:bad-fullspan} gives
$\Delta_b(u)=0$, which is condition~(2).

Conversely, if condition~(1) holds, then $b$ is bad by
\cref{lem:bad-plane} and \eqref{eq:rankcriterion}. If $u$ satisfies
\eqref{eq:involution} and $\Delta_b(u)=0$, then
\cref{lem:bad-fullspan} gives a nonzero pair whose image is contained
in the three-dimensional space $\Fix\rho_u$, so $b$ is bad by
\eqref{eq:rankcriterion}. The assertion about representatives is
part of \cref{lem:bad-plane}.
\end{proof}

\section{Bounds on the bad parameters}\label{sec:count}

We first bound the parameters in condition~(1) of \cref{thm:bad}.
Fix one of the $q^2+1$ representatives $\varepsilon$ just described.
Because $\varepsilon\notin\F_{q^2}$,
$\varepsilon^{q^5}-\varepsilon^q\ne0$. Equation
\eqref{eq:planecondition} determines
\[
 \eta=\frac{\tau-\varepsilon-\nu\varepsilon^q}
                {\nu(\varepsilon^{q^5}-\varepsilon^q)}
\]
from $(\varepsilon,\nu,\tau)$. There are at most
$(q^2+1)(q^4-1)q^4$ choices. Finally, $\eta$ determines $b$:
$\beta_0=\eta^q\ne0$ and
\begin{equation}\label{eq:recoverb}
 b=\frac{(1-\eta)^{q^5}}{\eta^q}.
\end{equation}
Thus at most $(q^2+1)(q^4-1)q^4$ parameters with
$\N_{q^{12}/q^4}(b)\ne-1$ satisfy condition~(1) of \cref{thm:bad}.

For condition~(2) of \cref{thm:bad}, we use the three-dimensional subspaces $V\le \F_{q^6}$
such that $\F_{q^2}V=\F_{q^6}$. For such a $V$, write $u_V$ for the map
determined by the involution fixing $V$, as before \cref{lem:bad-fullspan}.
Let
\[
 \mathcal B_V=\{b\in \F_{q^{12}}:\ c_{\alpha,\beta}(W_0)\subseteq V
                    \text{ for some }(\alpha,\beta)\ne(0,0)\}.
\]
By \cref{lem:bad-fullspan}, for every $V$ and every $b\in\F_{q^{12}}$,
\begin{equation}\label{eq:BV}
 b\in\mathcal B_V\quad\Longleftrightarrow\quad\Delta_b(u_V)=0.
\end{equation}
We call two such subspaces $V,V'$ equivalent under scalar
multiplication if $V'=eV$ for some $e\in\F_{q^6}^*$, and we call
$\{eV:e\in\F_{q^6}^*\}$ the scalar class of $V$. Rescaling
$(\alpha,\beta)$ shows that $\mathcal B_{eV}=\mathcal B_V$, so
subspaces in the same scalar class exclude the same parameters.

\begin{lemma}\label{lem:orbit-moment}
There are $q^4+q$ scalar classes of three-dimensional
$\F_q$-subspaces $V\le\F_{q^6}$ satisfying $\F_{q^2}V=\F_{q^6}$.
\end{lemma}
\begin{proof}
There are $|\GL_3(q^2)|$ ordered $\F_{q^2}$-bases of $\F_{q^6}$.
The $\F_q$-span of each such basis is a three-dimensional subspace
spanning $\F_{q^6}$ over $\F_{q^2}$. Conversely, every ordered
$\F_q$-basis of such a subspace is an $\F_{q^2}$-basis of
$\F_{q^6}$. Each subspace has exactly $|\GL_3(q)|$ ordered
$\F_q$-bases. Hence the number of subspaces is
\[
 \frac{|\GL_3(q^2)|}{|\GL_3(q)|}
 =q^3(q^3+1)(q^3+q^2+q+1).
\]

For each $V$, the set $\{e\in\F_{q^6}:eV\subseteq V\}$ is a
subfield containing $\F_q$. Indeed, it is closed under addition
and multiplication; if $e\ne0$ belongs to it, then $eV=V$, so
$e^{-1}$ also belongs to it. Since $V$ is a vector space over this
subfield and has dimension $3$ over $\F_q$, the subfield cannot
contain $\F_{q^2}$; hence it is either $\F_q$ or $\F_{q^3}$.
The second case occurs if and only if $V=e\F_{q^3}$ for some
$e\in\F_{q^6}^*$; these subspaces form one scalar class with
$q^3+1$ members. For every other $V$, the equality $eV=fV$
holds if and only if $e/f\in\F_q^*$, so its scalar class has
$(q^6-1)/(q-1)$ members. The number of classes is therefore
\[
 1+\frac{q^3(q^3+1)(q^3+q^2+q+1)-(q^3+1)}
         {(q^6-1)/(q-1)}
 =q^4+q.
\]
\end{proof}

\begin{lemma}\label{lem:BV-size}
For every three-dimensional $\F_q$-subspace $V\le\F_{q^6}$ with
$\F_{q^2}V=\F_{q^6}$,
\[
 |\mathcal B_V|\le q^8+q^4+1.
\]
\end{lemma}
\begin{proof}
Write $u=u_V$. The coefficient calculation \eqref{eq:alpha-Lambda}
in the proof of \cref{lem:bad-fullspan} shows that
$b\in\mathcal B_V$ if and only
if some $\beta\ne0$ satisfies $\Lambda_{u,b}(\beta^{q^3})=0$,
that is,
\[
 u_0(1-b^{q^3})\beta^{q^3}
 -u_2\beta^{q^7}+u_1b^{q^{11}}\beta^{q^{11}}=0.
\]
By the second equation in \eqref{eq:alpha-Lambda},
\[
 \alpha=u_2(\beta b)^{q^7}+u_1\beta^{q^{11}}
       -u_0\beta^{q^3}-u_1(\beta b)^{q^{11}}.
\]

Set $y=\beta^{q^3}$ and $z=b^{q^3}y$. The equation
$\Lambda_{u,b}(\beta^{q^3})=0$ becomes
\begin{equation}\label{eq:incidence}
 u_0z-u_1z^{q^8}=u_0y-u_2y^{q^4},\qquad y\ne0.
\end{equation}
Conversely, each solution $(y,z)$ determines the unique parameter
$b=(z/y)^{q^9}$ and the unique $\beta=y^{q^9}$, and hence a
nonzero pair $(\alpha,\beta)$ with image in $V$.
For every $\lambda\in\F_{q^4}^*$, the pair $(\lambda y,\lambda z)$
is another solution giving the same $b$. These $q^4-1$ pairs are
distinct, so $|\mathcal B_V|$ is at most the number of solutions
of \eqref{eq:incidence} divided by $q^4-1$.

If $\N_{q^6/q^2}(u_0)\ne\N_{q^6/q^2}(u_1)$, then
$z\mapsto u_0z-u_1z^{q^8}$ is bijective by
\cref{lem:dickson}. There is exactly one $z$ for every nonzero
$y$, so
\[
 |\mathcal B_V|\le\frac{q^{12}-1}{q^4-1}=q^8+q^4+1.
\]

Suppose now that
$\N_{q^6/q^2}(u_0)=\N_{q^6/q^2}(u_1)$. The first equation in \eqref{eq:involution} excludes
$u_0=u_1=0$, so both coefficients are nonzero.
By \cref{lem:dickson}, the map
$z\mapsto u_0z-u_1z^{q^8}$ is singular.
The ratio of any two nonzero kernel elements belongs to
$\F_{q^4}^*$, and every $\F_{q^4}$-multiple of a kernel element
again belongs to the kernel. Hence the kernel has $q^4$
elements and the image has dimension $2$ over $\F_{q^4}$.

If $\N_{q^6/q^2}(u_0)\ne\N_{q^6/q^2}(u_2)$, then
$y\mapsto u_0y-u_2y^{q^4}$ is bijective.
There is exactly one $y$ for every $z$, and this $y$ is zero
precisely for the $q^4$ elements satisfying
$u_0z-u_1z^{q^8}=0$. Consequently,
\[
 |\mathcal B_V|
 \le\frac{q^{12}-q^4}{q^4-1}
 =q^8+q^4.
\]

It remains to consider
\[
 \N_{q^6/q^2}(u_0)
 =\N_{q^6/q^2}(u_1)
 =\N_{q^6/q^2}(u_2).
\]
Here all three coefficients are nonzero. Both maps in
\eqref{eq:incidence} have kernels of dimension $1$ and images
of dimension $2$ over $\F_{q^4}$.

We claim that their images are distinct.
Otherwise, there would
exist $\gamma\in\F_{q^{12}}^*$ such that
\[
 \Tr_{q^{12}/q^4}\bigl(\gamma(u_0z-u_1z^{q^8})\bigr)=\Tr_{q^{12}/q^4}\bigl((\gamma u_0 -\gamma^{q^4}u_1^{q^4})z\bigr)=0,
\]
and \[
 \Tr_{q^{12}/q^4}\bigl(\gamma(u_0y-u_2y^{q^4})\bigr)= \Tr_{q^{12}/q^4}\bigl((\gamma u_0-\gamma^{q^8}u_2^{q^8})y\bigr)=0
\]
for every $y,z\in\F_{q^{12}}$.
Thus
\[
 u_0\gamma=u_1^{q^4}\gamma^{q^4},
 \qquad
 u_0\gamma=u_2^{q^8}\gamma^{q^8}=u_2^{q^2}\gamma^{q^8}.
\]
Therefore
\[
 \frac{u_0}{u_2^{q^2}}
 =\gamma^{q^8-1}
 =\bigl(\gamma^{q^4-1}\bigr)^{q^4+1}
 =\left(\frac{u_0}{u_1^{q^4}}\right)^{q^4+1}.
\]
Simplifying and raising to the $q^2$-th power yields
\[
 u_1^{1+q^4}=u_0u_2^{q^4}.
\]
It follows that
\[
 u_0u_1^{q^2}u_2^{q^4}
 =u_1^{1+q^2+q^4}
 =\N_{q^6/q^2}(u_1).
\]
Since the three norms are equal, the determinant in
\eqref{eq:dickson}, applied to $u$, is then
\[
 3\N_{q^6/q^2}(u_1)
 -\Tr_{q^6/q^2}\bigl(\N_{q^6/q^2}(u_1)\bigr)=0.
\]
This contradicts the invertibility of $u$, which follows from
$\rho_u^2=1$.

The two images are distinct $\F_{q^4}$-planes in $\F_{q^{12}}$, so
they meet in $q^4$ elements, each with $q^4$ preimages under either
map. Equation~\eqref{eq:incidence}, with $y=0$ allowed, has
$q^4\cdot q^4\cdot q^4=q^{12}$ solutions, $q^4$ of them with $y=0$,
and
\[
 |\mathcal B_V|
 \le\frac{q^{12}-q^4}{q^4-1}
 =q^8+q^4.
\]
\end{proof}

\section{Orthogonality and the parameter count}\label{sec:orthogonality}

In this section we show that $V$ and $V^\perp$ give the same set of
bad parameters, and we complete the proof of \cref{thm:four}.
Orthogonality refers to the nondegenerate pairing $(x,y)\mapsto\Tr_{q^6/q}(xy)$ on $\F_{q^6}$.

\begin{lemma}\label{lem:cofactors}
Let $u(z)=u_0z+u_1z^{q^2}+u_2z^{q^4}$ be an
$\F_{q^2}$-linear map on $\F_{q^6}$ such that $\rho_u^2=1$,
and put $d=\det_{\F_{q^2}}(u)$. Then $d\ne0$,
\[
 u^{-1}(z)=u_0^{q^3}z+u_1^{q^3}z^{q^2}+u_2^{q^3}z^{q^4},
\]
and
\begin{equation}\label{eq:cofactors}
\begin{aligned}
 d\,u_0^{q^3}&=u_0^{q^2}u_0^{q^4}-u_1^{q^2}u_2^{q^4},\\
 d\,u_1^{q^3}&=u_2u_2^{q^4}-u_1u_0^{q^4},\\
 d\,u_2^{q^3}&=u_1u_1^{q^2}-u_2u_0^{q^2}.
\end{aligned}
\end{equation}
\end{lemma}
\begin{proof}
The identity $\rho_u^2=1$ gives
$u^{-1}(z)=(u(z^{q^3}))^{q^3}$, proving the formula for the
inverse. Comparing coefficients in $u^{-1}(u(z))=z$ yields
\[
 \begin{pmatrix}
 u_0&u_2^{q^2}&u_1^{q^4}\\
 u_1&u_0^{q^2}&u_2^{q^4}\\
 u_2&u_1^{q^2}&u_0^{q^4}
 \end{pmatrix}
 \begin{pmatrix}u_0^{q^3}\\u_1^{q^3}\\u_2^{q^3}\end{pmatrix}
 =\begin{pmatrix}1\\0\\0\end{pmatrix}.
\]
The coefficient matrix is the transpose of the Dickson matrix
of $u$, so its determinant is $d\ne0$.
Cramer's rule gives \eqref{eq:cofactors}.
\end{proof}

\begin{lemma}\label{lem:pairing}
If $\dim_{\F_q}V=3$ and $\F_{q^2}V=\F_{q^6}$, then
$\F_{q^2}V^\perp=\F_{q^6}$ and
\[
 \mathcal B_{V^\perp}=\mathcal B_V.
\]
\end{lemma}
\begin{proof}
For $\kappa\in\F_{q^2}\setminus\F_q$ we have
$V\oplus\kappa V=\F_{q^6}$. Taking orthogonal complements gives
$V^\perp\cap\kappa^{-1}V^\perp=0$.
Since $V^\perp$ has dimension $3$ over $\F_q$, it too spans $\F_{q^6}$
over $\F_{q^2}$.

Write $u=u_V$. The adjoint $u^*$, defined by
$\Tr_{q^6/q}(u(x)y)=\Tr_{q^6/q}(xu^*(y))$, is
\[
 u^*(z)=u_0z+u_2^{q^2}z^{q^2}+u_1^{q^4}z^{q^4}.
\]
The adjoint of $\rho_u$ is
$\rho_u^*(z)=(u^*(z))^{q^3}$. Taking adjoints in $\rho_u^2=1$ gives
$(\rho_u^*)^2=1$, and $\rho_u^*$ is $q$-semilinear over $\F_{q^2}$
because $u^*$ is $\F_{q^2}$-linear.
Since $V=\ker(1-\rho_u)$,
$V^\perp=\operatorname{im}(1-\rho_u^*)$. Every element
$z=w-\rho_u^*(w)$ satisfies $\rho_u^*(z)=-z$. Hence
$\operatorname{im}(1-\rho_u^*)\subseteq\ker(1+\rho_u^*)$.
The first space has dimension three because $1-\rho_u^*$ has the
same rank as $1-\rho_u$. The second is the fixed space of the
$q$-semilinear involution $-\rho_u^*$, so it also has dimension three
by \cref{lem:fixed-spaces}.
The two spaces are therefore equal, so $-\rho_u^*$ is the involution
fixing $V^\perp$, and
\begin{equation}\label{eq:uperp}
 u_{V^\perp}(z)=-u_0^{q^3}z-u_2^{q^5}z^{q^2}-u_1^qz^{q^4}.
\end{equation}

Put $d=\det_{\F_{q^2}}(u)$. We show that
\[
 \Delta_b(u_{V^\perp})=-\frac{\Delta_b(u)}{d^2}.
\]
For this purpose, write
\[
 n_i=\N_{q^6/q^2}(u_i)\quad(0\le i\le2),\qquad
 t=u_0u_1^{q^2}u_2^{q^4},
\]
and put
\[
 \lambda=\Tr_{q^6/q^2}(t^{1+q^2})
             -n_0n_1-n_0n_2-n_1n_2\in\F_{q^2}.
\]
Recall that $d=n_0+n_1+n_2-\Tr_{q^6/q^2}(t)$.
Let $p_0,p_1,p_2$ denote the right-hand sides of
\eqref{eq:cofactors}, so $p_i=d\,u_i^{q^3}$.
Expanding their norms gives
\[
\begin{aligned}
 \N_{q^6/q^2}(p_0)
 &=n_0^2-n_0\Tr_{q^6/q^2}(t)
       +\Tr_{q^6/q^2}(t^{1+q^2})-n_1n_2
   =dn_0+\lambda,\\
 \N_{q^6/q^2}(p_1)
 &=n_2^2-n_2\Tr_{q^6/q^2}(t)
       +\Tr_{q^6/q^2}(t^{1+q^2})-n_0n_1
   =dn_2+\lambda,\\
 \N_{q^6/q^2}(p_2)
 &=n_1^2-n_1\Tr_{q^6/q^2}(t)
       +\Tr_{q^6/q^2}(t^{1+q^2})-n_0n_2
   =dn_1+\lambda.
\end{aligned}
\]
Similarly, the mixed product is
\[
\begin{aligned}
 p_0p_1^{q^2}p_2^{q^4}
 &=t^{q^2+q^4}-t^2+(n_0+n_1+n_2)t
                      -n_0n_1-n_0n_2-n_1n_2\\
 &=dt+\lambda.
\end{aligned}
\]
By \eqref{eq:uperp} and \eqref{eq:cofactors}, the coefficients
of $d\,u_{V^\perp}$ are $-p_0,-p_2^{q^2},-p_1^{q^4}$.
Since $d,\lambda\in\F_{q^2}\subseteq\F_{q^4}$,
substitution into \eqref{eq:Delta} yields
\[
\begin{aligned}
 d^3\Delta_b(u_{V^\perp})
 ={}&-d\Delta_b(u)\\
 &-\lambda\Bigl(
   \N_{q^{12}/q^4}(1-b^{q^3})
   +\N_{q^{12}/q^4}(b^{q^3})-1
   +\Tr_{q^{12}/q^4}((1-b^{q^3})b^{q^7})\Bigr).
\end{aligned}
\]
The expression in parentheses vanishes. Indeed, for
$a=b^{q^3}$, the cubic norm and trace satisfy
\[
\begin{aligned}
 \N_{q^{12}/q^4}(1-a)
 &=1-\Tr_{q^{12}/q^4}(a)
      +\Tr_{q^{12}/q^4}(a^{1+q^4})-\N_{q^{12}/q^4}(a),\\
 \Tr_{q^{12}/q^4}((1-a)a^{q^4})
 &=\Tr_{q^{12}/q^4}(a)-\Tr_{q^{12}/q^4}(a^{1+q^4}).
\end{aligned}
\]
Thus $d^3\Delta_b(u_{V^\perp})=-d\Delta_b(u)$.
Dividing by $d^3\ne0$ proves the claimed identity.
By \eqref{eq:BV}, it follows that
$\mathcal B_{V^\perp}=\mathcal B_V$.
\end{proof}

\begin{lemma}\label{lem:selfpaired}
Among the $q^4+q$ scalar classes in \cref{lem:orbit-moment},
at most $2q^3$ satisfy $V^\perp=aV$ for some
$a\in\F_{q^6}^*$.
\end{lemma}
\begin{proof}
Write $u=u_V$ and let $d=\det_{\F_{q^2}}(u)$.
If $a\in\F_{q^6}^*$, then $aV$ is the fixed space of the involution
$z\mapsto a\rho_u(a^{-1}z)$, hence
\[
u_{aV}(z)=a\,u(a^{-q^3}z)
         =a^{1-q^3}u_0z+a^{1-q^5}u_1z^{q^2}+a^{1-q}u_2z^{q^4}.
\]
Comparing with \eqref{eq:uperp}, the equality $V^\perp=aV$ is
equivalent to
\[
(u_0a)^{q^3}+u_0a=0,
\qquad
u_2=-a^{q-1}u_1^q;
\]
the remaining relation $a^{1-q^5}u_1=-u_2^{q^5}$ is the
$q^5$-th power of the second one.
The formula for $u^{-1}$ in \cref{lem:cofactors} gives
$\det_{\F_{q^2}}(u^{-1})=d^{q^3}=d^q$, so $d^{q+1}=1$.
Since $u_{aV}$ is obtained from $u$ by composing with the
multiplications by $a^{-q^3}$ and by $a$, we have
$\det_{\F_{q^2}}(u_{aV})=\N_{q^6/q^2}(a)^{1-q}d$, while
\eqref{eq:uperp} and \eqref{eq:dickson} give
$\det_{\F_{q^2}}(u_{V^\perp})=-d^q$. Hence
\[
\N_{q^6/q^2}(a)^{1-q}d=-d^q.
\]
Thus every such pair $(a,d)$ satisfies
\[
d^{q+1}=1,
\qquad
\N_{q^6/q^2}(a)^{q-1}=-d^{1-q}.
\]
There are exactly $q^6-1$ pairs satisfying these two equations.
Indeed, there are $q+1$ choices for $d$. For each of them,
$(-d^{1-q})^{q+1}=1$, so $-d^{1-q}$ has exactly $q-1$ preimages under
$x\mapsto x^{q-1}$ in $\F_{q^2}^*$. Since $\N_{q^6/q^2}$ maps
$\F_{q^6}^*$ onto $\F_{q^2}^*$ with fibres of size
$(q^6-1)/(q^2-1)$, the second equation has
$(q-1)(q^6-1)/(q^2-1)=(q^6-1)/(q+1)$ solutions $a$.

Fix such a pair $(a,d)$.
The equation $(u_0a)^{q^3}+u_0a=0$ gives exactly $q^3$
possibilities for $u_0$.
We show that each gives at most two possibilities for $u_1$.

The inverse coefficient identity
\[
du_0^{q^3}
=u_0^{q^2}u_0^{q^4}-u_1^{q^2}u_2^{q^4},
\]
after substituting $u_2=-a^{q-1}u_1^q$, using $d\in \F_{q^2}$, and raising to the
$q^4$-th power, gives
\[
u_1^{q^3+1}
=a^{q^2-q^3}\bigl(du_0^q-u_0^{1+q^2}\bigr).
\]
Put
\[
w=d^q u_1^{1+q^2+q^4}\in\F_{q^2}.
\]
Raising 
\[
du_1^{q^3}=u_2u_2^{q^4}-u_1u_0^{q^4}
\]
to the $q^3$-th power, using $u_2=-a^{q-1}u_1^q$, and multiplying by $u_1$ yields
\[
d^q u_1^2
=a^{q^2+q^4-q-q^3}u_1^{1+q^2+q^4}
-u_0^q u_1^{1+q^3}.
\]
Using $d^{q+1}=1$ and
$d\N_{q^6/q^2}(a)^{1-q}=-d^q$, this becomes
\[
u_1^2
=-a^{q^5-1}w-du_0^q u_1^{1+q^3}.
\]

The coefficient of $z$ in $\rho_u^2(z)=z$ gives
\[
u_0u_0^{q^3}+u_1u_2^{q^5}+u_2u_1^q=1.
\]
Since
\[
u_2=-a^{q-1}u_1^q,
\qquad
u_2^{q^5}=-a^{1-q^5}u_1,
\]
this equation becomes
\[
u_0u_0^{q^3}
-a^{1-q^5}u_1^2-a^{q-1}u_1^{2q}=1,
\]
or equivalently
\[
a^{1-q^5}u_1^2+a^{q-1}u_1^{2q}
=u_0u_0^{q^3}-1.
\]
Substituting
\[
u_1^2=-a^{q^5-1}w-du_0^q u_1^{1+q^3}
\]
and its $q$-th power gives
\[
\begin{aligned}
w+w^q={}&1-u_0u_0^{q^3}\\
&-da^{1-q^5}u_0^q u_1^{1+q^3}\\
&-d^qa^{q-1}u_0^{q^2}
   \bigl(u_1^{1+q^3}\bigr)^q.
\end{aligned}
\]

Also,
\[
w^{q+1}
=\N_{q^6/q^3}(u_1)^{1+q+q^2}.
\]
Thus both $w+w^q$ and $w^{q+1}$ are determined by $a,d,u_0$.
There are at most two possibilities for $w$, since it satisfies
the quadratic polynomial with this prescribed trace and norm.

Each value of $w$ determines $u_1^2$. Indeed, substituting
\[
u_1^{1+q^3}
=a^{q^2-q^3}\bigl(du_0^q-u_0^{1+q^2}\bigr)
\]
into the expression for $u_1^2$ gives
\[
u_1^2
=-a^{q^5-1}w
-da^{q^2-q^3}u_0^q
 \bigl(du_0^q-u_0^{1+q^2}\bigr).
\]
The right-hand side depends only on the fixed values
$a,d,u_0,w$.
Moreover, the definition of $w$, together with $d^{q+1}=1$,
gives
\[
u_1^{1+q^2+q^4}=dw.
\]
Thus both the square of $u_1$ and its norm to $\F_{q^2}$
are determined.

These two values determine $u_1$: if $u_1^2=0$, then $u_1=0$;
otherwise the ratio $x$ of two candidates satisfies $x^2=1$, so
$x\in\F_q$, and equality of their norms to $\F_{q^2}$ gives $x^3=1$;
hence $x=1$. Then $u_2=-a^{q-1}u_1^q$ is determined as well.

Consequently, there are at most
\[
2(q^6-1)\,q^3
\]
pairs $(a,V)$ satisfying $V^\perp=aV$.
Every scalar class admitting such a pair contributes exactly
$q^6-1$ pairs: if its scalar stabiliser has size $k_V$, it contains
$(q^6-1)/k_V$ subspaces, each admitting exactly $k_V$ values of $a$.
Division by $q^6-1$ proves the claim.
\end{proof}
\begin{proof}[Proof of \cref{thm:four}]
The map $V\mapsto V^\perp$ is an involution on scalar classes,
since $(eV)^\perp=e^{-1}V^\perp$.
There are $q^4+q$ classes by \cref{lem:orbit-moment},
and at most $2q^3$ are fixed by this involution by
\cref{lem:selfpaired}.
The number of orbits of the involution is therefore at most
\[
\frac{q^4+q+2q^3}{2}.
\]
By \cref{lem:pairing}, the classes in each orbit exclude
the same set of parameters.
Each such set has size at most $q^8+q^4+1$ by
\cref{lem:BV-size}.
Consequently, condition~(2) of \cref{thm:bad} holds for at most
\[
\frac{q^4+q+2q^3}{2}\,(q^8+q^4+1)
\]
parameters.

Condition~(1) holds for at most
$(q^2+1)(q^4-1)q^4$ parameters outside the norm exception.
There are exactly $q^8+q^4+1$ parameters satisfying
$\N_{q^{12}/q^4}(b)=-1$.
Discarding these parameters and those satisfying either
condition of \cref{thm:bad} leaves at least
\[
q^{12}
-\frac{q^4+2q^3+q+2}{2}\,(q^8+q^4+1)
-(q^2+1)(q^4-1)q^4
\]
good parameters, proving \eqref{eq:bound}.

To check positivity for $q\ge3$, divide this expression
by $q^{12}$. The result is
\[
1-\frac12(1+2q^{-1}+q^{-3}+2q^{-4})
             (1+q^{-4}+q^{-8})
 -(q^{-2}+q^{-4})(1-q^{-4}).
\]
For $q\ge3$, we have
\[
\frac12(1+2q^{-1}+q^{-3}+2q^{-4})
\le\frac{70}{81},
\]
and
\[
(q^{-2}+q^{-4})(1-q^{-4})
\le q^{-2}+q^{-4}\le\frac{10}{81}.
\]
Hence the normalised bound is at least
\[
1-\frac{70}{81}
  \left(1+\frac1{81}+\frac1{6561}\right)
-\frac{10}{81}
=\frac{821}{531441}>0.
\]
This proves existence for every $q\ge3$.
\Cref{app:q2} proves the remaining case $q=2$
by checking \eqref{eq:rankcriterion} directly.
\end{proof}

\section{The construction in dimensions coprime to six}\label{sec:coprime}

We now prove \cref{thm:coprime}. Together with \cref{thm:four}
and the dimension-three construction, this gives all the cases
needed in \cref{thm:main}. We again use \cref{prop:rank-code}.
Since $U_s$ is an $\F_{q^r}$-subspace, the rank condition becomes a
condition on sums of $\F_{q^r}$-subspaces
(\cref{lem:complement-criterion}).

For this construction, let $r\ge5$ satisfy
$\gcd(r,6)=1$, and put
\[
 J=\F_{q^r},\qquad E=\F_{q^{2r}},\qquad L=\F_{q^{6r}}.
\]
Thus $E=\F_{q^2}J$ and $L=\F_{q^6}J$. The fields
$\F_{q^6}$ and $J$ are linearly disjoint over $\F_q$.
In particular, an $\F_q$-basis of $J$ is an
$\F_{q^6}$-basis of $L$, and $L$ represents $V(r,q^6)$.
For an $\F_q$-subspace $C\le\F_{q^6}$, write
\[
 CJ=\langle C\rangle_J\le L.
\]
Linear disjointness gives $\dim_J(CJ)=\dim_{\F_q}C$.

Fix $\varepsilon\in\F_{q^6}\setminus\F_{q^2}$ and
$s\in J\setminus\F_q$. Since $\F_{q^6}/\F_{q^2}$ has
degree three, $1,\varepsilon,\varepsilon^2$ is a basis over
$\F_{q^2}$. The degrees of $\F_{q^6}/\F_{q^2}$ and
$E/\F_{q^2}$ are coprime, so these elements also form an
$E$-basis of $L$. In particular,
$\F_{q^6}\cap E=\F_{q^2}$.

We will use the fact that $s$ and $s^2$ have the same degree
over $\F_q$, and that this degree is at least five. Indeed,
$[\F_q(s):\F_q(s^2)]\le2$, and this degree divides the odd
integer $r$, so it is one. The common degree is a divisor of
$r$ greater than one. Since $\gcd(r,6)=1$, it is at least five.

For $x\in E$ write $\overline{x}=x^{q^r}$. This is the
nontrivial automorphism of $E/J$; it fixes $J$ and restricts
to $x\mapsto x^q$ on $\F_{q^2}$. On vectors and matrices,
the bar acts entrywise. Choose
$\kappa\in\F_{q^2}\setminus\F_q$ and put
\begin{equation}\label{eq:coprime-basis}
 D_s=1+\varepsilon s,\qquad
 N_s=\kappa+\overline{\kappa}\varepsilon s.
\end{equation}
Here $E=J\oplus\kappa J=J\oplus\overline{\kappa}J$.

\begin{lemma}\label{lem:coprime-dimension}
The subspace $U_s$ in \eqref{eq:coprime-family} satisfies
\[
 U_s=D_sJ\oplus N_sJ.
\]
It has dimension $2r$ over $\F_q$ and spans $L$ over $\F_{q^6}$.
\end{lemma}
\begin{proof}
Write $x=a+\kappa b$ with $a,b\in J$. Then
\[
 x+\varepsilon s\overline{x}
 =a(1+\varepsilon s)+b(\kappa+\overline{\kappa}\varepsilon s)
 =aD_s+bN_s,
\]
so $U_s=D_sJ+N_sJ$. If a nonzero $x\in E$ satisfies
$x+\varepsilon s\overline{x}=0$, then
$\varepsilon=-x/(s\overline{x})\in E$, contrary to
$\F_{q^6}\cap E=\F_{q^2}$.
Hence the map defining $U_s$ is injective, its image has
dimension $2r$ over $\F_q$, and the sum is direct.
Finally, injectivity gives $D_s\ne0$, and
\[
 \langle U_s\rangle_{\F_{q^6}}
 \supseteq\langle D_sJ\rangle_{\F_{q^6}}=D_sL=L.
\]
\end{proof}

The subspace $U_s$ depends only on $\varepsilon s$.
The choice of $\kappa$ specifies a $J$-basis and does not
change the subspace.

\begin{remark}
If $s$ has degree $r$ over $\F_q$, write its minimal polynomial as
\[
 m_s(X)=X^r+m_{r-1}X^{r-1}+\cdots+m_1X+m_0\in\F_q[X].
\]
In the $\F_{q^6}$-basis $1,s,\ldots,s^{r-1}$ of $L$,
the elements of $U_s$ have coordinates
\begin{equation}\label{eq:coprime-coordinates}
 \begin{pmatrix}
 x_0-\varepsilon m_0x_{r-1}^q\\
 x_1+\varepsilon(x_0^q-m_1x_{r-1}^q)\\
 \vdots\\
 x_{r-1}+\varepsilon(x_{r-2}^q-m_{r-1}x_{r-1}^q)
 \end{pmatrix},
 \qquad x_0,\ldots,x_{r-1}\in\F_{q^2}.
\end{equation}
Indeed, substitute $x=\sum_{i=0}^{r-1}x_is^i$ in
\eqref{eq:coprime-family}, use $\overline{x_i}=x_i^q$, and
reduce $s^r$ by $m_s(s)=0$.
\end{remark}

We next express the rank condition in \cref{prop:rank-code}
using subspaces of $L$ over $J$.

\begin{lemma}\label{lem:complement-criterion}
Let $w_1,w_2\in L\setminus\{0\}$ and suppose that $w_1J+w_2J$ has dimension two
over $J$. Every nonzero $\F_{q^6}$-linear functional has
rank at least four over $\F_q$ on $w_1J+w_2J$ if and only if
\[
 w_1(CJ)\cap w_2(CJ)=\{0\}
\]
for every three-dimensional $\F_q$-subspace $C\le\F_{q^6}$.
\end{lemma}
\begin{proof}
Every $\F_{q^6}$-linear functional on $L$ has the form
$z\mapsto\Tr_{q^{6r}/q^6}(az)$ for a unique $a\in L$.
Since $w_1J$ spans $L$ over $\F_{q^6}$, its restriction to
$w_1J+w_2J$ is nonzero whenever $a\ne0$.
In the coordinates $w_1x_1+w_2x_2$, the restriction is
\[
 M_a:J^2\longrightarrow\F_{q^6},\qquad
 (x_1,x_2)\longmapsto\Tr_{q^{6r}/q^6}(a(w_1x_1+w_2x_2)).
\]
Use the trace pairing on $\F_{q^6}/\F_q$ and the pairing
\[
 ((x_1,x_2),(x_1',x_2'))\longmapsto\Tr_{q^r/q}(x_1x_1'+x_2x_2')
\]
on $J^2$. The adjoint is
\[
 \Phi_a:\F_{q^6}\longrightarrow J^2,\qquad
 y\longmapsto
 \bigl(\Tr_{q^{6r}/q^r}(aw_1y),\Tr_{q^{6r}/q^r}(aw_2y)\bigr).
\]
Indeed, transitivity of the trace gives
\[
\begin{aligned}
 \Tr_{q^6/q}(yM_a(x_1,x_2))
 &=\Tr_{q^{6r}/q}(ay(w_1x_1+w_2x_2))\\
 &=\Tr_{q^r/q}\bigl(x_1\Tr_{q^{6r}/q^r}(aw_1y)
                  +x_2\Tr_{q^{6r}/q^r}(aw_2y)\bigr).
\end{aligned}
\]
The pairings are nondegenerate, so
\[
 \operatorname{rank}_{\F_q}M_a
 =\operatorname{rank}_{\F_q}\Phi_a
 =6-\dim_{\F_q}\ker\Phi_a.
\]
Thus a nonzero functional has rank at most three if and only
if there are $a\in L\setminus\{0\}$ and a three-dimensional
$\F_q$-subspace $C\le\F_{q^6}$ with $C\subseteq\ker\Phi_a$.
By $J$-linearity of the trace, this is equivalent to
\[
 0\ne a\in\bigl(w_1(CJ)+w_2(CJ)\bigr)^\perp,
\]
where the orthogonal complement is taken for the trace pairing
on $L/J$. Such an $a$ exists if and only if the sum is proper.
Each summand has dimension three over $J$, whereas $L$ has
dimension six. The sum is therefore proper if and only if the
intersection is nonzero.
\end{proof}

\section{Complementarity}\label{sec:complementarity}

Let $C\le\F_{q^6}$ have dimension three over $\F_q$.
We prove that $D_s(CJ)\cap N_s(CJ)=\{0\}$.
As in the proof of \cref{thm:bad}, we distinguish the two
possible dimensions of $\F_{q^2}C$.

\begin{lemma}\label{lem:coprime-rankthree}
If $\F_{q^2}C=\F_{q^6}$, then
\[
 D_s(CJ)\cap N_s(CJ)=\{0\}.
\]
\end{lemma}
\begin{proof}
Choose an $\F_q$-basis $c_1,c_2,c_3$ of $C$.
It is an $\F_{q^2}$-basis of $\F_{q^6}$ and an $E$-basis
of $L$. Put
\[
 \varphi(z)=c_1z_1+c_2z_2+c_3z_3\qquad(z\in E^3),
\]
and let $A\in\GL_3(q^2)$ be the matrix of multiplication by
$\varepsilon$ in this basis, so that
$\varepsilon\varphi(z)=\varphi(Az)$.
Here $CJ=\varphi(J^3)$, and the coordinatewise involution
$z\mapsto\overline z$ fixes $J^3$.

Suppose that $D_sx+N_sy=0$ with $x,y\in CJ$.
Write $x=\varphi(\mathbf x)$ and $y=\varphi(\mathbf y)$ for
$\mathbf x,\mathbf y\in J^3$, and set $z=\mathbf x+\overline{\kappa}\mathbf y$. Since
$\overline z=\mathbf x+\kappa\mathbf y$, the equation becomes
\begin{equation}\label{eq:coprime-semilinear}
 \overline z+sAz=0.
\end{equation}
Applying the involution and substituting gives
\begin{equation}\label{eq:coprime-normmatrix}
 (I_3-s^2\overline A A)z=0.
\end{equation}
Put $B=\overline A A$. Then
\[
 \overline B=A\overline A=ABA^{-1},
\]
so the characteristic polynomial of $B$ is fixed by the
involution. Its coefficients lie in $\F_{q^2}$ and hence
in $\F_q$. Consequently,
\[
 P_A(X)=\det(I_3-XB)\in\F_q[X]
\]
has degree at most three and constant term one.
The degree of $s^2$ over $\F_q$ is at least five, so
$P_A(s^2)\ne0$. Equation~\eqref{eq:coprime-normmatrix}
therefore gives $z=0$. Since
$E=J\oplus\overline{\kappa}J$, we obtain $\mathbf x=\mathbf y=0$, and
hence $x=y=0$.
\end{proof}

Suppose now that $\F_{q^2}C$ has dimension two, and put
\[
 W_\varepsilon=\F_{q^2}\oplus\F_{q^2}\varepsilon.
\]
By \cref{lem:trace-planes}, a scalar multiple of
$\F_{q^2}C$ is $W_\varepsilon$. Multiplying $C$ by
$h\in\F_{q^6}^*$ does not change the required condition,
since
\[
 D_s((hC)J)=hD_s(CJ),\qquad
 N_s((hC)J)=hN_s(CJ).
\]
We may therefore assume that $C\subseteq W_\varepsilon$.
For $\xi\in\F_{q^2}$ put
\begin{equation}\label{eq:normal-finite}
 C_\xi=\langle1,\xi+\varepsilon,
                  \kappa(\xi+\varepsilon)\rangle_{\F_q},
\end{equation}
and put
\begin{equation}\label{eq:normal-infinity}
 C_\infty=\langle1,\kappa,\varepsilon\rangle_{\F_q}.
\end{equation}

\begin{lemma}\label{lem:normalforms}
Every three-dimensional $\F_q$-subspace of $W_\varepsilon$
is a scalar multiple, by an element of $\F_{q^2}^*$, of
one of the spaces in \eqref{eq:normal-finite} and
\eqref{eq:normal-infinity}.
\end{lemma}
\begin{proof}
Write the elements of $W_\varepsilon$ as $x_1+x_2\varepsilon$
with $x_1,x_2\in\F_{q^2}$. An $\F_q$-hyperplane has an equation
\[
 \Tr_{q^2/q}(\lambda_1x_1+\lambda_2x_2)=0
\]
for some $(\lambda_1,\lambda_2)\ne(0,0)$. Multiplication of the
hyperplane by $k\in\F_{q^2}^*$ replaces the coefficients
by $(\lambda_1/k,\lambda_2/k)$.

Choose $\theta\in\F_{q^2}^*$ with
$\Tr_{q^2/q}(\theta)=0$. In characteristic two one can
take $\theta=1$; otherwise one can take
$\theta=\kappa-\kappa^q$. The nonzero map
$z\mapsto\Tr_{q^2/q}(\theta z)$ has kernel $\F_q$.

If $\lambda_1\ne0$, scale by $k=\lambda_1/\theta$ and set
$\xi=-\lambda_2/\lambda_1$. The equation is then
$x_1-\xi x_2\in\F_q$. Writing $x_1=t_0+\xi x_2$ and
$x_2=t_1+\kappa t_2$ with $t_0,t_1,t_2\in\F_q$ gives $C_\xi$.
If $\lambda_1=0$, scale by $k=\lambda_2/\theta$. The equation
becomes $x_2\in\F_q$, giving $C_\infty$.
\end{proof}

\begin{lemma}\label{lem:coprime-ranktwo}
For every $\xi\in\F_{q^2}\cup\{\infty\}$,
\[
 D_s(C_\xi J)\cap N_s(C_\xi J)=\{0\}.
\]
\end{lemma}
\begin{proof}
First let $\xi\in\F_{q^2}$. We have
\[
 C_\xi J=\{a+\xi e+e\varepsilon:a\in J,\ e\in E\}.
\]
Suppose that $D_sX+N_sY=0$, where
\[
 X=a_1+\xi e_1+e_1\varepsilon,\qquad
 Y=a_2+\xi e_2+e_2\varepsilon,
 \qquad a_1,a_2\in J,\quad e_1,e_2\in E.
\]
Comparing coefficients in the $E$-basis
$1,\varepsilon,\varepsilon^2$ gives
\begin{align}
 e_1+\overline{\kappa}e_2&=0,\label{eq:coprime-r2a}\\
 a_1+\kappa a_2+\xi(e_1+\kappa e_2)&=0,\label{eq:coprime-r2b}\\
 e_1+\kappa e_2+s(a_1+\overline{\kappa}a_2)&=0.\label{eq:coprime-r2c}
\end{align}
The first equation comes from the coefficient of
$\varepsilon^2$. The coefficient of $\varepsilon$ is
\[
 e_1+\kappa e_2+s(a_1+\overline{\kappa}a_2)
       +\xi s(e_1+\overline{\kappa}e_2),
\]
which gives the last equation after using
\eqref{eq:coprime-r2a}.

Substituting \eqref{eq:coprime-r2c} in
\eqref{eq:coprime-r2b} yields
\begin{equation}\label{eq:coprime-ac}
 (1-\xi s)a_1+(\kappa-\overline{\kappa}\xi s)a_2=0.
\end{equation}
Put $f_1=1-\xi s$ and $f_2=\kappa-\overline{\kappa}\xi s$.
Here $f_1\ne0$, since $J\cap\F_{q^2}=\F_q$ and
$s\notin\F_q$. The elements $f_1,f_2$ are dependent over $J$
if and only if $f_2/f_1$ is fixed by the involution, that is,
$f_1\overline{f_2}-\overline{f_1}f_2=0$. Expanding gives
\[
\begin{aligned}
 f_1\overline{f_2}-\overline{f_1}f_2
 &=(1-\xi s)(\overline{\kappa}-\kappa\overline{\xi}s)
   -(1-\overline{\xi}s)(\kappa-\overline{\kappa}\xi s)\\
 &=(\overline{\kappa}-\kappa)
       \bigl(1-\N_{q^2/q}(\xi)s^2\bigr).
\end{aligned}
\]
The first factor is nonzero. The second is one if $\xi=0$;
otherwise its vanishing would imply $s^2\in\F_q$.
Thus $f_1,f_2$ are independent over $J$, and
\eqref{eq:coprime-ac} gives $a_1=a_2=0$.
Equations~\eqref{eq:coprime-r2a} and
\eqref{eq:coprime-r2c} then give $e_1=e_2=0$.

For $\xi=\infty$, write
\[
 C_\infty J=\{e+a\varepsilon:e\in E,\ a\in J\}.
\]
Let $X=e_1+a_1\varepsilon$ and $Y=e_2+a_2\varepsilon$ satisfy
$D_sX+N_sY=0$. The coefficient of $\varepsilon^2$ gives
$a_1+\overline{\kappa}a_2=0$. Since $a_1,a_2\in J$ and
$\overline{\kappa}\notin J$, we have $a_1=a_2=0$.
The other coefficients give
\[
 e_1+\kappa e_2=0,\qquad s(e_1+\overline{\kappa}e_2)=0,
\]
and hence $e_1=e_2=0$.
\end{proof}

\begin{theorem}\label{thm:complement}
For every three-dimensional $\F_q$-subspace $C\le\F_{q^6}$,
\[
 D_s(CJ)\oplus N_s(CJ)=L.
\]
\end{theorem}
\begin{proof}
If $\F_{q^2}C=\F_{q^6}$, apply
\cref{lem:coprime-rankthree}. Otherwise, rescale into
$W_\varepsilon$ and apply \cref{lem:normalforms,lem:coprime-ranktwo}.
In each case the intersection is trivial. Both summands have
dimension three over $J$, so their sum is $L$.
\end{proof}

\begin{proof}[Proof of \cref{thm:coprime}]
By \cref{lem:coprime-dimension}, $U_s$ has dimension $2r$
over $\F_q$ and spans $L$ over $\F_{q^6}$.
\Cref{thm:complement} and \cref{lem:complement-criterion}, applied
with $w_1=D_s$ and $w_2=N_s$, show that every nonzero $\F_{q^6}$-linear functional on $L$ has
rank at least four on $U_s$. The result now follows from
\cref{prop:rank-code}.
\end{proof}

\section{The associated MRD codes}\label{sec:codes}

Let $U$ satisfy the conditions of \cref{prop:rank-code}, let $G$ have
as columns an $\F_q$-basis of $U$, and let $H$ be a generator matrix of
$\mathcal C_U^\perp$. Let $U'$ be the $\F_q$-span of the columns of
$H$. Then $U'$ is, up to $\GL_r(q^6)$-equivalence, the Delsarte dual
of $U$ \cite[Section~3]{CMPZ2021}; in the terminology of
\cite[Definition~2.17]{MNT2023}, it is the system associated with
$\mathcal C_U^\perp$. Indeed, put $W=\F_q^{2r}\subseteq\F_{q^6}^{2r}$
and $\Gamma=\ker G$, kernels being taken on column vectors. Then
$\Gamma=\mathcal C_U^\perp$, and $\Gamma\cap W=\{0\}$ because the
columns of $G$ are $\F_q$-independent; $G$ maps $(W+\Gamma)/\Gamma$
onto $U$. Moreover $\Gamma^\perp=\mathcal C_U=\ker H$, so $H$ maps
$(W+\Gamma^\perp)/\Gamma^\perp$ onto $U'$. A nonzero
$x\in\F_q^{2r}$ with $Hx=0$ would be a codeword of $\mathcal C_U$ of
rank one, so $\dim_{\F_q}U'=2r$. Since $H$ has rank $r$, $U'$ spans
$\F_{q^6}^r$, and $\mathcal C_{U'}=\mathcal C_U^\perp$ is a
$[2r,r,4]_{q^6/q}$ MRD code. By \cref{prop:rank-code}, $U'$ is maximum
$2$-scattered in $V(r,q^6)$: for these parameters Delsarte duality
preserves the ambient dimension. This duality should not be confused
with the map $V\mapsto V^\perp$ of \cref{sec:orthogonality}.

By \cref{prop:rank-code}, both constructions give MRD codes.
We first make the code from $U_s$ explicit and determine its
duality. Fix an $\F_q$-basis $\beta_1,\ldots,\beta_r$ of $J$.
With respect to the basis
\[
 D_s\beta_1,\ldots,D_s\beta_r,
 N_s\beta_1,\ldots,N_s\beta_r
\]
of $U_s$, put
\begin{equation}\label{eq:coprime-code}
 \mathcal C_s=\{\mathbf w_a:a\in L\},
\end{equation}
where
\begin{equation}\label{eq:coprime-word}
\begin{aligned}
 \mathbf w_a={}&\bigl(\Tr_{q^{6r}/q^6}(aD_s\beta_1),\ldots,
                        \Tr_{q^{6r}/q^6}(aD_s\beta_r),\\
 &\hspace{9mm}\Tr_{q^{6r}/q^6}(aN_s\beta_1),\ldots,
                        \Tr_{q^{6r}/q^6}(aN_s\beta_r)\bigr).
\end{aligned}
\end{equation}

\begin{proposition}\label{prop:isodual}
The code $\mathcal C_s$ is a $[2r,r,4]_{q^6/q}$ MRD code
equivalent to its dual by a matrix in $\GL_{2r}(q)$.
The Delsarte dual of $U_s$ is equivalent to $U_s$ under an element of
$\GL_r(q)\le\GL_r(q^6)$; explicitly, it is $PU_s$, where $P$ is the
Gram matrix defined in the proof.
\end{proposition}
\begin{proof}
The parameters follow from \cref{thm:coprime,prop:rank-code}.
Let $T\in\F_q^{r\times r}$ be the matrix of multiplication
by $s$ on $J$ in the basis $\beta_1,\ldots,\beta_r$.
This is also an $\F_{q^6}$-basis of $L$. In this basis,
multiplication by $D_s$ and $N_s$ has matrices
\[
 M_D=I_r+\varepsilon T,\qquad
 M_N=\kappa I_r+\overline{\kappa}\varepsilon T.
\]
Their columns give the chosen $\F_q$-basis of $U_s$, so
$G=(M_D\mid M_N)$ generates $\mathcal C_s$.
The two matrices are invertible and commute. Hence
\[
 H=(-M_N^{\mathsf T}\mid M_D^{\mathsf T})
\]
has rank $r$ and satisfies $GH^{\mathsf T}=0$, so it
generates $\mathcal C_s^\perp$.

Put
\[
 P=(\Tr_{q^r/q}(\beta_i\beta_j))_{i,j=1}^r.
\]
The trace pairing is nondegenerate, so $P\in\GL_r(q)$.
Multiplication by $s$ is self-adjoint for this pairing,
giving $T^{\mathsf T}P=PT$. Consequently,
\[
 M_D^{\mathsf T}=PM_DP^{-1},\qquad
 M_N^{\mathsf T}=PM_NP^{-1}.
\]
It follows that
\begin{equation}\label{eq:isodual}
 H=PGQ,\qquad
 Q=\begin{pmatrix}0&P^{-1}\\-P^{-1}&0\end{pmatrix}
   \in\GL_{2r}(q).
\end{equation}
Thus $\mathcal C_s^\perp=\mathcal C_sQ$.

As recalled above, the $\F_q$-span of the columns of $H$
represents the Delsarte dual of $U_s$.
Equation~\eqref{eq:isodual}
identifies this span with $PU_s$, proving the last assertion.
\end{proof}

For the dimension-four construction, choosing an $\F_q$-basis
$x_1,\ldots,x_8$ of $W_0$ gives the code
\[
 \mathcal C_b=\{(c_{\alpha,\beta}(x_1),\ldots,
                    c_{\alpha,\beta}(x_8)):
                   \alpha,\beta\in\F_{q^{12}}\}.
\]
For every good $b$, this is an $[8,4,4]_{q^6/q}$ MRD code by
\cref{prop:rank-code}.

\begin{corollary}\label{cor:codes-all-r}
For every prime power $q$ and every integer $r\ge3$, there
exists an $\F_{q^6}$-linear $[2r,r,4]_{q^6/q}$ MRD code.
\end{corollary}
\begin{proof}
Apply \cref{prop:rank-code} to the subspaces in
\cref{thm:main}.
\end{proof}

In particular, the constructions give $[8,4,4]_{q^6/q}$,
$[10,5,4]_{q^6/q}$ and $[14,7,4]_{q^6/q}$ MRD codes for
every $q$. The explicit equivalence in \eqref{eq:isodual}
applies to the latter two families.

\section{Open problems}\label{sec:conclusions}

The proof of \cref{thm:four} is not constructive: for general $q$ it
does not give an explicit $b$ such that $U_b$ is maximum
$2$-scattered. It would be interesting to find a simple sufficient
condition on $b$, or to describe the intersections of the sets
$\mathcal B_V$ more precisely.

The parameters with $\N_{q^{12}/q^4}(b)=-1$ were discarded in the
count. For $q=2$, the computation in \cref{app:q2} shows that
exactly $132$ of these $273$ parameters give maximum
$2$-scattered subspaces; for $q=3$, $4200$ of the $6643$ parameters
of norm $-1$ are good. It remains to characterise the good
parameters of norm $-1$ for general $q$.

We do not know how many of the subspaces $U_b$ are pairwise
inequivalent, nor whether, for $q$ an odd power of $2$, some of them
are equivalent to the subspaces of \cite{BGGM2025}. Invariants of the
associated codes, such as their generalised rank weights
\cite{MNT2023,BGM2026}, could be used to distinguish them.

For $r\in\{4,5\}$ the subspaces $U_b$ and $U_s$ are not direct sums of
maximum $2$-scattered subspaces of smaller dimensions, since neither
$4$ nor $5$ is a sum of integers greater than or equal to $3$.
For the family $U_s$, which depends only on the product
$\varepsilon s$, it would be interesting to determine when two
products give equivalent subspaces, and whether, for $r\ge7$, $U_s$
can be equivalent to a direct sum of maximum $2$-scattered subspaces
of smaller dimensions (for $r=7$, of dimensions $3$ and $4$).

\section*{Acknowledgements}
The authors thank the Italian
National Group for Algebraic and Geometric Structures and their
Applications (GNSAGA-INdAM) for its support. A. Neri is supported by the University of Naples Federico II through the FRA2024 Project CUP E63C26001530001  ``NHACT -- New Horizons in Algebraic Coding Theory''.

\section*{Data and code availability}
The Python code used for the computations is listed in
\cref{app:q2}.

\appendix
\section{The case \texorpdfstring{$q=2$}{q=2}}\label{app:q2}

Let
\[
 \F_{2^{12}}\cong\F_2[\omega]/(\omega^{12}+\omega^{11}+\cdots+\omega+1),\qquad
 b=\omega^{11}+\omega^{10}+\omega^8+\omega^7+\omega^6+\omega^5+\omega^4+\omega^2.
\]
The defining polynomial is $\Phi_{13}$, which is irreducible over
$\F_2$ because the multiplicative order of $2$ modulo $13$ is $12$.
The coefficient string of $b$, highest degree first, is
\texttt{110111110100}.

The verification below only uses arithmetic in $\F_{2^{12}}$ and
Gaussian elimination over $\F_2$. The integer $\sum_{i=0}^{11}c_i2^i$
represents $\sum_{i=0}^{11}c_i\omega^i$.
A basis of $W_0$ in this representation is
\[
 10,\ 36,\ 68,\ 151,\ 279,\ 514,\ 1040,\ 2071.
\]
Choose one representative of each of the $65$ cosets of $\F_{2^6}^*$ in
$\F_{2^{12}}^*$. Every hyperplane can be represented uniquely by a pair
$(\alpha,\beta)$ in which either $\alpha$ is one of these representatives
and $\beta\in \F_{2^{12}}$, or $\alpha=0$ and $\beta$ is one of them.
Multiplication of a pair by an element of $\F_{2^6}^*$ does not change
the rank of its map. Thus there are exactly
$65\cdot4096+65=266305$ ranks to check.

The resulting distribution of ranks is
\[
\begin{array}{c|rrr}
\dim_{\F_2}c_{\alpha,\beta}(W_0)&4&5&6\\\hline
\text{number of hyperplanes}&2635&57630&206040
\end{array}
\]
In particular, every rank is at least $4$, as required. These
numbers agree with the rank distribution of an MRD code of
$6\times8$ matrices over $\F_2$ with minimum rank distance $4$
\cite{Delsarte1978}, after division by $2^6-1$ to count nonzero
codewords up to $\F_{2^6}^*$-scaling.

Applying the same rank test to all $4096$ parameters gives exactly
$1056$ values of $b$ for which $U_b$ is maximum $2$-scattered:
$924$ parameters
satisfy $\N_{2^{12}/2^4}(b)\ne1$, and $132$ of the $273$
parameters with $\N_{2^{12}/2^4}(b)=b^{273}=1$ also work
(in characteristic two, these are the parameters of norm $-1$
excluded in \cref{thm:bad}).
To reduce the computation, observe
that the coordinatewise Frobenius map $(x,y)\mapsto(x^2,y^2)$
sends $U_b$ to $U_{b^2}$, since squaring preserves $W_0$.
This map preserves $2$-scatteredness, so it suffices to test one
representative of each of the $352$ Frobenius orbits in
$\F_{2^{12}}$.
These computations can be reproduced with the following Python script,
which uses no external packages. In the exhaustive search, a test
stops as soon as a rank below four is found.

\begin{lstlisting}[language=Python]
MODULUS = (1 << 13) - 1
ORDER = 1 << 12
B = int("110111110100", 2)

def mul(x, y):
    z = 0
    while y:
        if y & 1:
            z ^= x
        y >>= 1
        x <<= 1
        if x & ORDER:
            x ^= MODULUS
    return z

def power(x, n):
    y = 1
    while n:
        if n & 1:
            y = mul(y, x)
        x = mul(x, x)
        n >>= 1
    return y

def rank(vectors):
    pivots = [0] * 12
    r = 0
    for v in vectors:
        while v:
            j = v.bit_length() - 1
            if pivots[j]:
                v ^= pivots[j]
            else:
                pivots[j] = v
                r += 1
                break
    return r

assert pow(2, 12, 13) == 1
assert all(pow(2, d, 13) != 1 for d in (1, 2, 3, 4, 6))
fq6_elements = [x for x in range(ORDER) if power(x, 64) == x]
W = [x for x in range(ORDER)
     if x ^ power(x, 16) ^ power(x, 256) == 0]
assert len(fq6_elements) == 64 and len(W) == 256
basis = []
for x in W:
    if rank(basis + [x]) > len(basis):
        basis.append(x)
assert len(basis) == 8
trace = [x ^ power(x, 64) for x in range(ORDER)]
seen, reps = set(), []
for x in range(1, ORDER):
    if x not in seen:
        reps.append(x)
        seen.update(mul(x, e) for e in fq6_elements if e)
assert len(reps) == 65 and len(seen) == 4095
def rank_distribution(b, stop_on_bad=False):
    graph = [power(x, 2) ^ mul(b, power(x, 32)) for x in basis]
    odd = [[trace[mul(beta, y)] for y in graph]
           for beta in range(ORDER)]
    counts = {i: 0 for i in range(7)}
    for alpha in reps + [0]:
        even = [trace[mul(alpha, x)] for x in basis]
        betas = range(ORDER) if alpha else reps
        for beta in betas:
            r = rank([x ^ y for x, y in zip(even, odd[beta])])
            counts[r] += 1
            if stop_on_bad and r < 4:
                return counts
    return counts

counts = rank_distribution(B)
assert sum(counts.values()) == 266305
assert all(counts[i] == 0 for i in range(4))
assert counts[4] == 2635
assert counts[5] == 57630
assert counts[6] == 206040
print({r: n for r, n in counts.items() if n})

remaining = set(range(ORDER))
good, orbit_count = 0, 0
norm_one, good_norm_one = 0, 0
while remaining:
    b = min(remaining)
    orbit, c = set(), b
    while c not in orbit:
        orbit.add(c)
        c = power(c, 2)
    assert orbit <= remaining
    has_norm_one = power(b, 273) == 1
    if has_norm_one:
        norm_one += len(orbit)
    counts = rank_distribution(b, stop_on_bad=True)
    if all(counts[i] == 0 for i in range(4)):
        assert sum(counts.values()) == 266305
        good += len(orbit)
        if has_norm_one:
            good_norm_one += len(orbit)
    remaining -= orbit
    orbit_count += 1
assert orbit_count == 352 and norm_one == 273
assert good == 1056 and good_norm_one == 132
assert good - good_norm_one == 924
print("All parameters:", ORDER, "; good:", good)
print("Good with norm one:", good_norm_one)
print("Good with norm different from one:", good - good_norm_one)
\end{lstlisting}

\end{document}